\documentclass[oneside]{amsart}
\usepackage{amsmath,amssymb}
\usepackage[margin=2.5cm]{geometry}
\usepackage[magyar,british]{babel}
\usepackage[utf8]{inputenc}
\usepackage{latexsym}
\usepackage{amsfonts}
\usepackage{amssymb}
\usepackage{amsthm}
\usepackage{setspace}
\usepackage{faktor}
\usepackage{bm}
\usepackage{lmodern}
\usepackage{graphicx}
\usepackage{indentfirst}
\usepackage[hidelinks]{hyperref}
\usepackage{natbib}
\usepackage{epstopdf}
\usepackage{float}
\usepackage{xcolor}
\usepackage{enumitem}
\usepackage{caption}
\usepackage[labeled,resetlabels]{multibib}
\usepackage{etoolbox}

\newcites{D,ND}{,}

\def\restrict#1{\raise-.5ex\hbox{\ensuremath|}_{#1}}
\makeatletter
\let\old@schapter\@schapter
\AtBeginDocument{
	\def\@makechapterhead#1{\vspace*{30\p@}
		{\parindent \z@ \raggedright \normalfont
			\interlinepenalty\@M
			\Huge \bfseries
			\ifnum \c@secnumdepth >\m@ne
			\thechapter.\
			\fi
		#1\par \vskip 40\p@}}\def\@makeschapterhead#1{
	\vspace*{30\p@}{\parindent \z@ \raggedright
		\normalfont
		\interlinepenalty\@M
\Huge \bfseries #1\par \vskip 40\p@}}}
\DeclareRobustCommand{\qed}{%
\ifmmode 
\else \leavevmode\unskip\penalty9999 \hbox{}\nobreak\hfill
\fi
\quad\hbox{$\square$}}

\newcommand{\prooofname}{Proof:}

\newtheorem{theorem}{Theorem}[section]
\newtheorem{definition}[theorem]{Definition}
\newtheorem{proposition}[theorem]{Proposition}
\newtheorem{corollary}[theorem]{Corollary}
\newtheorem{lemma}[theorem]{Lemma}
\newtheorem{remark}[theorem]{Remark}

\newcommand{\Bin}{B_{\mathrm{in}}}
\newcommand{\Bout}{B_{\mathrm{out}}}
\newcommand{\Din}{D_{\mathrm{in}}}
\newcommand{\Dout}{D_{\mathrm{out}}}
\newcommand{\K}{\mathbf{K}}

\newcommand{\U}{\mathcal{U}}
\newcommand{\tb}{\mathrm{tb}}
\newcommand{\rot}{\mathrm{rot}}

\newcommand{\tw}{\mathrm{tw}}
\newcommand{\rtw}{\mathrm{rtw}}

\renewcommand{\wr}{\mathrm{wr}}

\renewcommand{\L}{\mathcal{L}}

\renewcommand{\U}{\mathcal{U}}

\newcommand\restr[2]{{
\left.\kern-\nulldelimiterspace 
#1 
\vphantom{\big|} 
\right|_{#2} 
}}

\makeatother
\begin{document}
	
	\title{Bounds on Legendrian Double Twist Knots}
	
	\author[Vikt\'oria F\"oldv\'ari]{Vikt\'oria F\"oldv\'ari}
	\address{
		Institute of Mathematics,
		ELTE E\"otv\"os Lor\'and University,
		Budapest, Hungary
	}
	\email{foldvari.viktoria@ttk.elte.hu}
	
	\author[Vera V\'ertesi]{Vera V\'ertesi}
	\address{
		Faculty of Mathematics,
		University of Vienna,
		 Vienna, Austria
	}
	\email{vera.vertesi@univie.ac.at}
	
	
	
	\begin{abstract}
	We study Legendrian realizations of double twist knots in the
	standard tight contact $3$-sphere. For the family $K(4,m)$, we give upper bounds on the number of oriented and unoriented Legendrian isotopy classes of maximal
	Thurston--Bennequin representatives. The proof uses bypass
	techniques in convex surface theory and the classification
	of tight contact structures.
	\end{abstract}
	\maketitle

\section{Introduction}

Legendrian knot theory is richer than smooth knot theory in the sense
that a fixed smooth knot type may admit several distinct Legendrian
isotopy classes. Besides the underlying smooth knot type,
Legendrian knots have two classical numerical invariants, the
Thurston--Bennequin invariant $\tb$ and the rotation number $\rot$.
A knot type is called Legendrian simple if these invariants determine
its Legendrian realizations up to Legendrian isotopy, and Legendrian
non-simple otherwise. Understanding the Legendrian realizations of a
fixed smooth knot type is a fundamental problem in
Legendrian knot theory.
In general, even deciding whether a given knot type is Legendrian
simple is difficult. Complete Legendrian classifications are known,
for example, for the unknot \cite{EliashbergFraser}, torus knots and
the figure-eight knot \cite{EtnyreHonda01}, and twist knots
\cite{ENV}. The classification of twist knots by Etnyre, Ng and
V\'ertesi is of particular relevance here, as it exhibits the rich
behavior that can already occur within this family.

Double twist knots form a family of two-bridge
knots, for which we use the standard diagram shown in Figure \ref{fig:double_twist}, where the upper twist region contains $k$ half twists and the lower box contains $m$ half twists. We denote the resulting knot type by $K(k,m)$. With this
convention, $K(k,m)$ corresponds to the continued fraction expansion
$[k,m-1,1]$. Double twist knots are natural generalizations of the
Eliashberg--Chekanov twist knots, which are recovered in the case
$k=2$.

\begin{figure}[h!]
	\centering
	\includegraphics[width=0.25\textwidth]{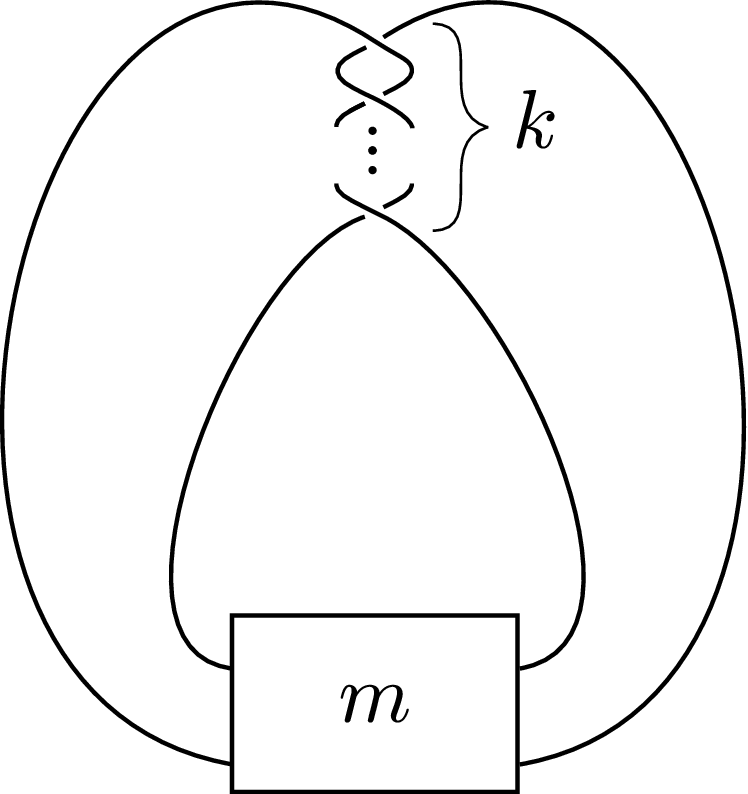}
	\caption{The double twist knot $K(k,m)$}
	\label{fig:double_twist}
\end{figure}

In this paper we concentrate on the family $K(4,m)$, building on research presented in the first author's PhD thesis \cite{FoldvariThesis}.
The cases $|m|<2$ are already understood: $K(4,0)$ is the unknot, while
$K(4,\pm1)$ are torus knots and hence Legendrian simple
\cite{EliashbergFraser,EtnyreHonda01}. For the upper bound, we therefore restrict attention to $|m|\geq2$. Our main result gives the
following upper bound.

\begin{theorem}\label{thm:main}
	For $|m|\geq 2$, the number of unoriented Legendrian isotopy classes of maximal
	Thurston--Bennequin representatives of the double twist knot
	$K(4,m)$ in $(S^3,\xi_{\mathrm{std}})$ is at most
	$$
	\begin{cases}
		6, & \textrm{if } m\geq 2,\\[2mm]
		6\cdot \left\lfloor \dfrac{(|m|+2)^2}{4}\right\rfloor, & \textrm{if } m\leq -2.
	\end{cases}
	$$
\end{theorem}

For oriented Legendrian knots, we obtain the following upper bound.

\begin{corollary}\label{cor:oriented}
	For $|m|\geq2$, the number of oriented Legendrian isotopy
	classes of maximal Thurston--Bennequin representatives
	of the double twist knot $K(4,m)$ in
	$(S^3,\xi_{\mathrm{std}})$ is at most
	$$
	\begin{cases}
		10, & \textrm{if } m\geq2 \text{ and } m \text{ is even},\\[2mm]
		12, & \textrm{if } m>2 \text{ and } m \text{ is odd},\\[2mm]
		10\cdot\left\lfloor\dfrac{(|m|+2)^2}{4}\right\rfloor,
		& \textrm{if } m\leq-2.
	\end{cases}
	$$
\end{corollary}

The proof is based on convex surface theory. We cut $S^3$ and the
double twist knot along a $2$-sphere separating the two twist regions.
After making the sphere convex and normalizing its dividing set, we
study the two sides separately. The normalized dividing set provides
common boundary data for comparing the pieces. We obtain upper bounds
for the possible inner Legendrian pieces and tight contact structures
on the outer complement, which combine to give Theorem \ref{thm:main}.
The resulting bound is uniform for $m\geq2$ and quadratic in $|m|$
for $m\leq-2$.

For completeness, we also record the maximal Thurston--Bennequin
invariant of $K(4,m)$ in our convention. This also follows from
Ng's computation for two-bridge links \cite{NgMaxTb}.

\begin{proposition}\label{thm:max_tb}
	The maximal Thurston--Bennequin invariant of the double twist knot
	$K(4,m)$ in $(S^3,\xi_{\mathrm{std}})$ is
	$$
	\tb_{\max}(K(4,m))=
	\begin{cases}
		-1-m, & \textrm{if } m\geq-2 \text{ and } m \text{ is even},\\
		-6, & \textrm{if } m=-1,\\
		-9-m, & \textrm{if } m\geq1 \text{ and } m \text{ is odd},\\
		1, & \textrm{if } m<-2 \text{ and } m \text{ is even},\\
		-7, & \textrm{if } m<-2 \text{ and } m \text{ is odd}.
	\end{cases}
	$$
\end{proposition}

Section \ref{sec:prelim_convex} collects the necessary background from
Legendrian knot theory and convex surface theory. Section \ref{sec:upper_bound} contains the proofs of our main results. 

\section{Preliminaries}\label{sec:prelim_convex}

We briefly recall some notions and results of Legendrian knot theory and convex surface theory that we use in the proofs. For a more detailed introduction the reader is referred to \cite{EtnyreConvex, EtnyreHonda01, Giroux91, HondaNotes, HondaOnTheI, ctct_book, FoldvariThesis}.

Throughout the paper, we work in the standard tight contact
three-sphere $(S^3,\xi_{\mathrm{std}})$. We use the same notation
$\xi_{\mathrm{std}}$ for the standard contact structure on
$\mathbb{R}^3$,
$$
\xi_{\mathrm{std}}=\ker(dz-y\,dx).
$$
Recall that removing a point from $(S^3,\xi_{\mathrm{std}})$ gives a
contact manifold contactomorphic to $(\mathbb{R}^3,\xi_{\mathrm{std}})$.

We work with the standard front projection to the $xz$-plane, and use calligraphic letters for Legendrian knots, while regular letters for smooth knot types.

The contact planes along $\L$ determine a framing $f_{TB}$ of
$\L$: it is represented by any nowhere-vanishing vector field in
$\xi_{\mathrm{std}}|_{\L}$ transverse to $T\L$. We call this the
Thurston--Bennequin framing. A Seifert surface of $\L$ determines
the Seifert framing $f_S$. For two framings
$f_1$ and $f_2$ of $\L$, let $\tw(f_1,f_2)$ denote the twisting
of $f_1$ with respect to $f_2$. Then
$$
\tb(\L)=\tw(f_{TB},f_S).
$$

If $\Sigma$ is an oriented Seifert surface of $\L$, then a
trivialization of $\xi_{\mathrm{std}}|_\Sigma$ induces a
trivialization of $\xi_{\mathrm{std}}$ along $\L$. The winding number
of $T\L$ with respect to this trivialization is the
\emph{rotation number} of $\L$, denoted by $\rot(\L)$.

Both invariants can be computed directly from a front projection.

\begin{proposition}\cite{ctct_book}
	\label{prop:classical_from_front}
	Let $\L\subset(\mathbb{R}^3,\xi_{\mathrm{std}})$ be an oriented
	Legendrian knot. Let $\wr(\L)$ denote the writhe of its front
	projection, let $c$ be the number of cusps, and let $c_d$ and $c_u$
	denote the numbers of downward and upward oriented cusps,
	respectively. Then
	\[
	\tb(\L)=\wr(\L)-\frac{1}{2}c,
	\qquad
	\rot(\L)=\frac{1}{2}(c_d-c_u).
	\]
\end{proposition}

A \emph{stabilization} of an oriented Legendrian knot is obtained by
replacing a small segment of its front projection by a zig-zag with
two cusps. If the cusps are oriented downward, we call the
stabilization positive and denote it by $S_+(\L)$; if they are
oriented upward, we call it negative and denote it by $S_-(\L)$.
The inverse operation is called a \emph{destabilization}.

\begin{proposition}\label{prop:stab_invts}
	Stabilization changes the classical invariants according to
	$$
	\tb(S_\pm(\L))=\tb(\L)-1,
	\qquad
	\rot(S_\pm(\L))=\rot(\L)\pm1.
	$$
\end{proposition}

In particular, every Legendrian representative with maximal
Thurston--Bennequin invariant is non-destabilizable.

Normal rulings provide a convenient combinatorial tool for detecting
maximal Thurston--Bennequin representatives.

\begin{definition}
	A \emph{ruling} of a Legendrian front diagram is a one-to-one correspondence between the left and right cusps of the diagram together with a decomposition of the diagram as the union of pairs of paths joining the corresponding left and right cusps and satisfying the following conditions: 
	\begin{itemize}
		\item Any two paths meet only at cusps or crossings.
		\item Every path is smooth away from the double points and is monotone in the	$x$-coordinate.
		\item The interiors of two paths joining corresponding cusps are disjoint, that is, they meet only at the cusps and their union bounds a topological disk called a \emph{ruling disk}.
		\item At a crossing either the two paths cross transversely or switch from one arc to the other so that one path lies entirely above the other. The latter type of crossing is called a \emph{switch}, which we denote on the knot diagram by a black dot at the crossing, see Figure \ref{fig:switch}.
	\end{itemize}
	
	If near each switch the ruling disks are locally either disjoint or nested, in other words, if all the switch crossings are as in Figure \ref{fig:switch}, then we call the configuration a \emph{normal ruling}. Figure \ref{fig:disallowed_switches} shows switches which are disallowed in a normal ruling.
\end{definition}

A ruling of a diagram can be given by indicating which paths form a pair, by coloring the ruling disks that appear or simply by denoting which crossings are switches. Observe that this latter way is sufficient, because starting at any left cusp we can follow the path (uniquely determined by the switch and non-switch crossings) until the corresponding right cusp, and then follow the other uniquely determined path back to the same left cusps. Therefore, from now on we will describe rulings only by marking which crossings of the diagram are switches.

\begin{figure}[!h]
	\centering
	\includegraphics[width=0.4\textwidth]{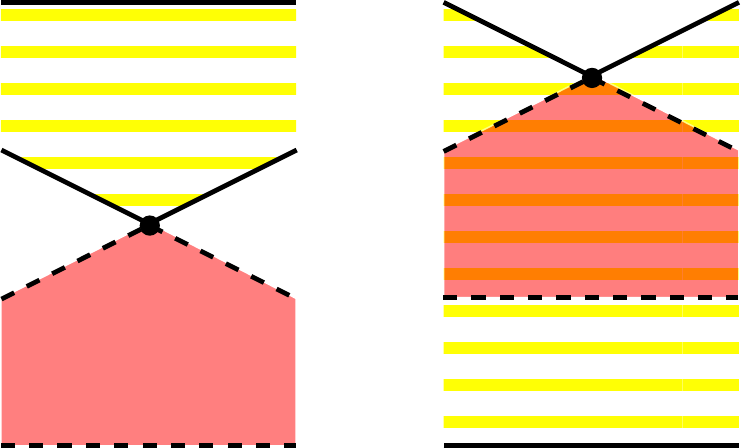}
	\caption{Allowed switches of a normal ruling. Here we only illustrate a neighborhood of a switch crossing. The style of the arcs (continuous or dashed) indicate which two paths form a pair, i.e. share cusp endpoints. The two ruling disks bounded by the two pairs of paths are colored by red and yellow striped. Orange stripes on red background express the intersection of the two ruling disks. In this picture, near the switch, ruling disks are either disjoint or nested.}
	\label{fig:switch}
\end{figure}

\begin{figure}[!h]
	\centering
	\includegraphics[width=0.4\textwidth]{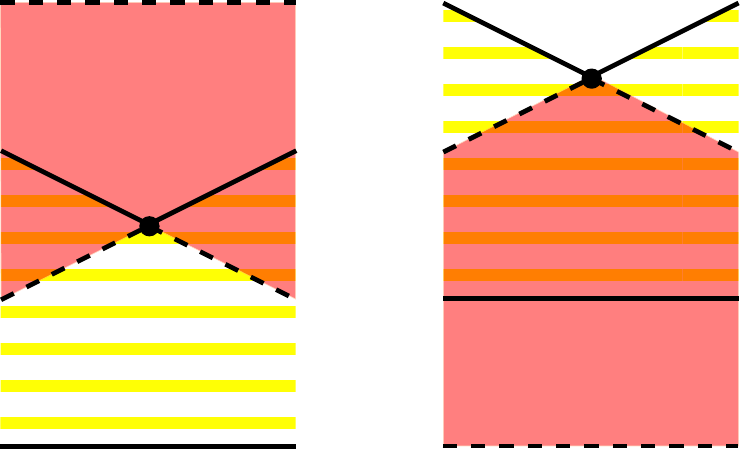}
	\caption{Switches which are disallowed in a normal ruling. We only illustrate a neighborhood of a switch crossing. The style of the arcs (continuous or dashed) indicate which two paths form a pair, i.e. share cusp endpoints. The two ruling disks bounded by the two pairs of paths are colored by red and yellow striped. Orange stripes on red background express the intersection of the two ruling disks. In this picture, ruling disks are neither disjoint nor nested near the switch.}
	\label{fig:disallowed_switches}
\end{figure}

We rely on the following result of Rutherford to verify that the fronts
appearing below have maximal Thurston--Bennequin invariant.

\begin{proposition}\cite{Rutherford2006}
	\label{prop:ruling_maxtb}
	If the front projection of a Legendrian knot admits a normal ruling,
	then it realizes the maximal Thurston--Bennequin invariant within its
	smooth knot type.
\end{proposition}

	We now turn to the convex surface techniques used in the proof.

\begin{definition}
	An oriented, embedded surface $\Sigma$ in a contact 3-manifold is called a \emph{convex surface} if there exists a vector field transverse to $\Sigma$ whose flow preserves the contact structure. We call such a vector field a \emph{contact vector field}.
\end{definition}

For a Legendrian curve $\gamma\subset\Sigma$, we write
$\tw(\gamma,\Sigma)$ for the twisting of its contact framing
with respect to the framing induced by $\Sigma$.
We will use the fact that closed embedded surfaces can be
$C^\infty$-perturbed to be convex \cite{Giroux91}. More generally,
a compact surface $\Sigma$ with Legendrian boundary can be made convex
provided that $\tw(\gamma,\Sigma)\leq 0$ for every boundary component $\gamma$ \cite{Kanda1998,HondaOnTheI}.

The main object associated to a convex surface is its dividing set.

\begin{definition}
	Consider a contact vector field $v$ on a convex surface $\Sigma \subset (Y, \xi)$. The \emph{dividing set} on $\Sigma$ is $\Gamma_\Sigma=\{x\in \Sigma | v(x)\in \xi(x)\}$. In other words, if $\xi$ is the kernel of the 1-form $\alpha$ then $\Gamma_\Sigma=\{x\in \Sigma | \alpha(v(x))=0\}$.
\end{definition}

The dividing set $\Gamma_\Sigma$ is a finite union of disjoint smooth
curves and properly embedded arcs, and its isotopy type is independent
of the choice of the transverse contact vector field \cite{Giroux91}. Therefore, it is legitimate that we often refer to the dividing set as \emph{dividing curves}. The dividing set carries essential information of the contact structure on a neighborhood of the convex surface. 

On a convex torus in standard form, we use the usual terminology of
\emph{Legendrian divides} and \emph{Legendrian rulings}.
Giroux's Flexibility Theorem \cite{Giroux91} allows us to choose the Legendrian rulings
conveniently -- we need the following standard consequence:

\begin{lemma}\cite{HondaOnTheI}\label{lemma:ruling_barmi}
	Let $T$ be a convex torus in $(Y,\xi)$. Consider the identification of $T$ with $\faktor{\mathbb{R}^2}{\mathbb{Z}^2}$ with coordinates $(x,y)$. Let the Legendrian rulings be the curves $y=rx+b$, and the Legendrian divides $y=\frac{k}{2n}$. Then we can modify the slopes of the rulings by a $C^0$-small perturbation near the divides from $r\neq0$ to any other $r'\neq 0$, including $r'=\infty$.
\end{lemma}

Studying the dividing set, we can also tell whether there exists a tight contact neighborhood of the convex surface:

\begin{theorem}[Giroux's Criterion, Giroux, \cite{GirouxStructures}]
	\label{thm:GirouxCriterion}
	Let $\Sigma\subset(Y,\xi)$ be a closed convex surface or a compact
	convex surface with Legendrian boundary, and let $\Gamma_\Sigma$
	denote its dividing set. If $\Sigma\neq S^2$, then $\Sigma$ has a
	tight neighborhood if and only if $\Gamma_\Sigma$ has no
	homotopically trivial closed components. If $\Sigma=S^2$, then
	$\Sigma$ has a tight neighborhood if and only if $\Gamma_\Sigma$
	is connected.
\end{theorem}

Convex surfaces also allow suitable curves to be made Legendrian,
through the Legendrian Realization Principle of Kanda \cite{KandaClassification}, generalized by Honda \cite{HondaOnTheI}.

\begin{theorem}[Legendrian Realization Principle, Kanda, \cite{KandaClassification}, Honda, \cite{HondaOnTheI}]\label{thm:LegendrianRealizationPrinciple}
	Let $\Sigma\subset (Y, \xi)$ be a convex surface and $C\subset \Sigma$ a properly embedded multicurve which is \emph{nonisolating}, that is, transverse to $\Gamma$ and the closure of every component of $\Sigma\setminus C$ intersects $\Gamma$. Then $C$ can be made Legendrian. In other words, there exists an isotopy $\phi_t$, $t\in [0,1]$ of $\Sigma$ such that $\phi_0(\Sigma)=\Sigma$, the dividing curves $\phi_1(\Gamma)=\Gamma_{\phi_1(\Sigma)}$, the surface $\phi_t(\Sigma)$ is convex for all $t$, and $\phi_1(C)$ is Legendrian.
\end{theorem}

Neighborhoods of Legendrian knots admit a standard contact model:

\begin{theorem}[Standard Contact Neighborhood Theorem,
	\cite{Geiges}]
	\label{thm:standard_neighborhood}
	Every neighborhood $U$ of a Legendrian knot $\L\subset(Y,\xi)$
	contains a tubular neighborhood $N(\L)$ which is contactomorphic to
	a neighborhood of the Legendrian core
	$$
	L_0=\{((x,y),z)\in\mathbb{R}^2\times S^1\mid x=y=0\}
	$$
	in
	$$
	\left(
	\mathbb{R}^2\times S^1,
	\ker\bigl(\sin(2\pi z)\,dx+\cos(2\pi z)\,dy\bigr)
	\right),
	$$
	with $\L$ mapped to $L_0$. We call $N(\L)$ a
	\emph{standard contact neighborhood} of $\L$.
\end{theorem}

We choose a standard contact neighborhood $N(\L)$ such that
$\partial N(\L)$ is convex and in standard form. Let $\mu$ and
$\lambda$ denote the meridian and the Seifert longitude on
$\partial N(\L)$, respectively, and identify
$\partial N(\L)$ with $\mathbb{R}^2/\mathbb{Z}^2$ so that
$\mu$ corresponds to $(1,0)$ and $\lambda$ to $(0,1)$.
We define the slope of a curve representing $p\mu+q\lambda$ to be
$q/p$. If $\tb(\L)=n$, then the dividing set on
$\partial N(\L)$ consists of two parallel curves of slope $1/n$
\cite[Section 1.1]{EtnyreHonda05cabling}.

To glue convex pieces back together, we will repeatedly use Honda's \emph{edge-rounding} procedure \cite{HondaOnTheI}: when two convex surfaces meet along a common
Legendrian boundary component, the edge can be rounded so that the
dividing curves connect according to the local model shown in
Figure \ref{fig:edge_rounding}.

\begin{theorem}[Edge Rounding, Honda, \cite{HondaOnTheI}]\label{thm:edge_rounding}
	For $i=1,2$, let $\Sigma_i$ be a convex surface with dividing set $\Gamma_i$ and assume that $\partial \Sigma_1=\partial \Sigma_2$ is Legendrian. Using the standard local model around $\partial \Sigma_i$ in $\mathbb{R}^3$ we can write $\Sigma_1=\{(x,y,z)| x=0, y\geq 0\}$ and $\Sigma_2=\{(x,y,z)| y=0, x\geq 0\}$. Then we can glue $\Sigma_1$ to $\Sigma_2$ by rounding the edge $\partial \Sigma_1=\partial \Sigma_2$ to get a smooth surface $\Sigma$ so that the dividing curves $\Gamma_i$ connect up to form the dividing set on $\Sigma$ according to Figure \ref{fig:edge_rounding}. We can model the edge rounding by replacing a small neighborhood $\nu(\partial\Sigma)$ with the intersection \mbox{$\nu(\partial\Sigma)\cap \{(x,y,z) | (x-\rho)^2+(y-\rho)^2=\rho^2\}$} for a suitable $\rho\in \mathbb{R}$.
\end{theorem}

\begin{figure}[h!]
	\centering
	\includegraphics[width=0.6\textwidth]{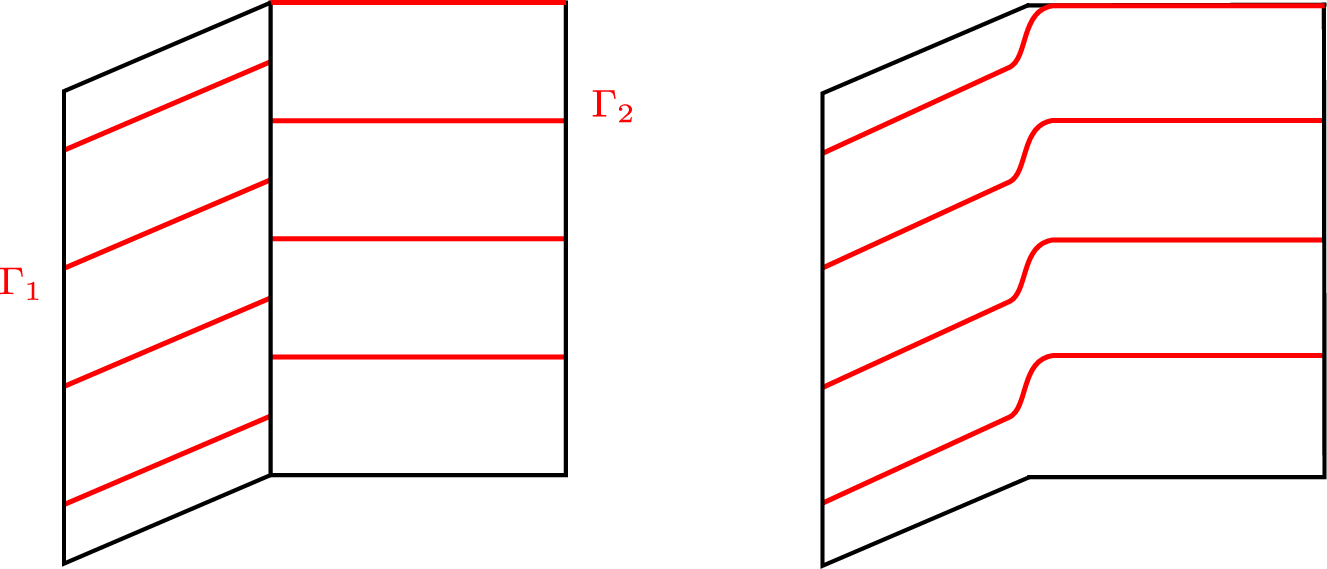}
	\caption{Edge rounding between two convex surfaces}
	\label{fig:edge_rounding}
\end{figure}

Eliashberg \cite{Eliashberg20years, Eliashberg1993} classified tight contact structures on $B^3$, $S^3$, $S^2\times S^1$ and $\mathbb{R}^3$. Of these, the most significant for our purpose is his theorem on the 3-ball.

\begin{theorem}[Eliashberg, \cite{Eliashberg20years}]\label{thm:B3unique}
	Let $\xi$ be a contact structure on a neighborhood of $\partial B^3$ such that $\partial B^3$ is convex with a connected dividing set. Then $\xi$ uniquely extends to a tight contact structure on $B^3$ up to an isotopy fixing the boundary.
\end{theorem}

We can consider the above $(B^3, \xi)$ as a building block of contact 3-manifolds. Many classification results are based on the idea of cutting a 3-manifold into 3-balls and examine the dividing set on the boundaries of each piece -- an idea we will also use in the proof of our main theorem.

There is a fundamental operation on convex surfaces providing an effective technique to investigate contact structures. 

Let $\Sigma$ be a convex surface with dividing set $\Gamma_\Sigma$.
A \emph{bypass} for $\Sigma$ is a convex half-disk attached to
$\Sigma$ along a Legendrian arc $\alpha$ which intersects
$\Gamma_\Sigma$ in three points, as shown in Figure
\ref{fig:bypass}. We refer to $\alpha$ as the \emph{attaching arc}.
For the standard definition and conventions, see
\cite{HondaOnTheI}.

\begin{figure}[h!]
	\centering
	\includegraphics[width=0.2\textwidth]{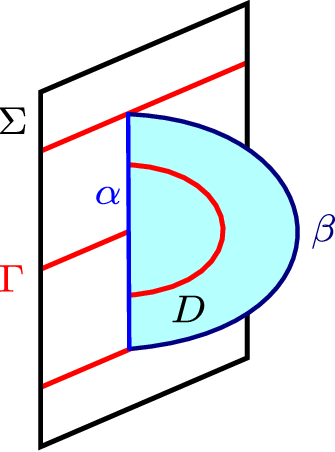}
	\caption{Bypass disk for $\Sigma$}
	\label{fig:bypass}
\end{figure}

Bypasses provide useful tools to find destabilizations of Legendrian knots. A bypass disk determines a destabilization of its Legendrian attaching arc. More precisely, the following is true, see \cite[Proposition 1.]{Laura10}.

\begin{proposition}\cite{Laura10}\label{prop:destab}
	Let $\Sigma\subset(Y, \xi)$ be a convex surface. Suppose that there is a bypass $D$ for $\Sigma$ along a Legendrian attaching arc $\alpha$. Let $\partial D=\alpha\cup \beta$. Then the Legendrian arc $\beta$ is a destabilization of $\alpha$.
\end{proposition}

Given a convex surface with a bypass disk, we can always push the surface across the disk so that the resulting surface is again convex and its dividing set changes in a nice way, according to the following theorem of Honda \cite{HondaOnTheI}:

\begin{theorem}[Bypass Attachment, Honda, \cite{HondaOnTheI}]\label{thm:bypass_attachment}
	Assume $\Sigma\subset (Y, \xi)$ is a convex surface with a bypass disk $D$ along $\alpha\subset \Sigma$. Then there exists a neighborhood $\Sigma\times [0,1]$ of $\Sigma\cup D$ such that $\Sigma=\Sigma\times \{0\}$, the surface $\Sigma\times \{1\}$ is convex and the dividing sets $\Gamma_{\Sigma}$ and $\Gamma_{\Sigma\times \{1\}}$ are related as in Figure \ref{fig:bypass_Gamma}. 
\end{theorem}

\begin{figure}[h!]
	\centering
	\includegraphics[width=0.45\textwidth]{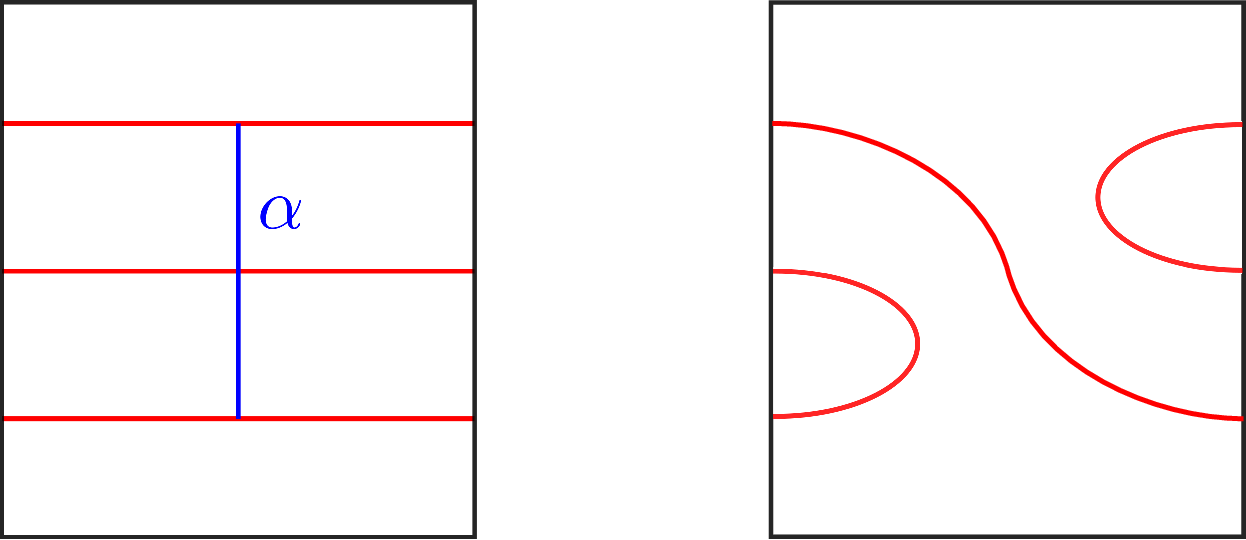}
	\caption{Local change of the dividing set after a bypass attachment. A portion of $\Sigma$ with attaching arc $\alpha$ is on the left, while the corresponding portion of $\Sigma\times \{1\}$ is on the right.}
	\label{fig:bypass_Gamma}
\end{figure}

\begin{definition}
	We say that in Theorem \ref{thm:bypass_attachment}, the convex surface $\Sigma\times \{1\}$ is obtained from $\Sigma$ by a \emph{bypass attachment from the front.} We can consider a \emph{bypass attachment from the back} if we change the orientation of the interval $[0,1]$ in the theorem above and mirror Figure \ref{fig:bypass_Gamma}, see Figure \ref{fig:bypass_back}
\end{definition}

\begin{figure}[h!]
	\centering
	\includegraphics[width=0.45\textwidth]{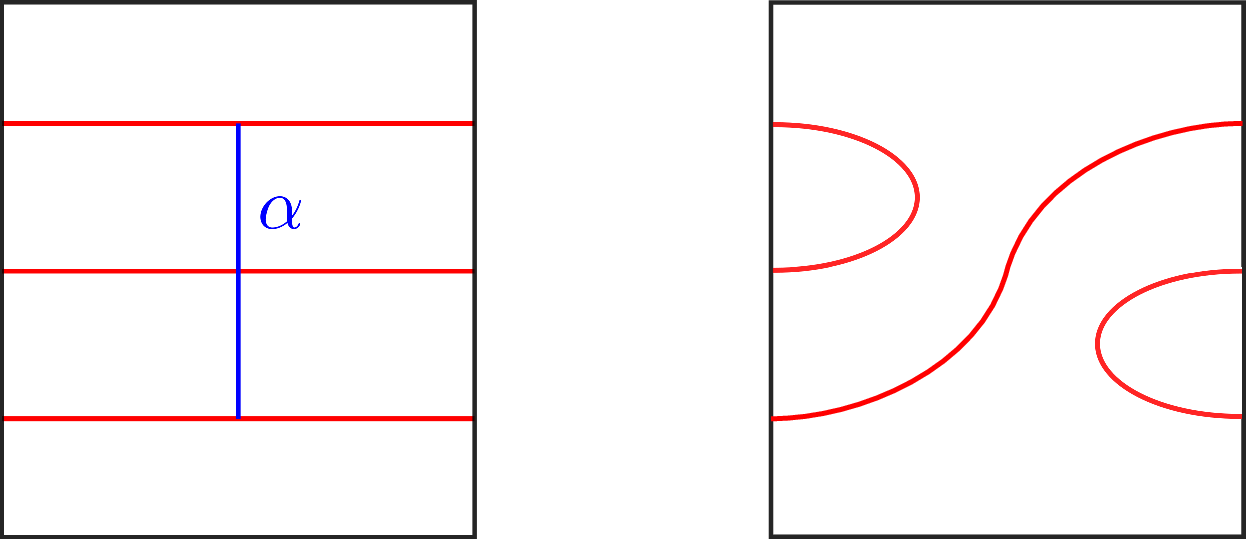}
	\caption{Local change of the dividing set after a bypass attachment \mbox{from the back}}
	\label{fig:bypass_back}
\end{figure}

Not every formal bypass attachment can occur in a tight contact
manifold: a bypass attachment which
produces a homotopically trivial component of the dividing set is
impossible by Giroux's Criterion. We will refer to such a bypass as
\emph{disallowed}.


\begin{figure}[h!]
	\centering
	\includegraphics[width=0.4\textwidth]{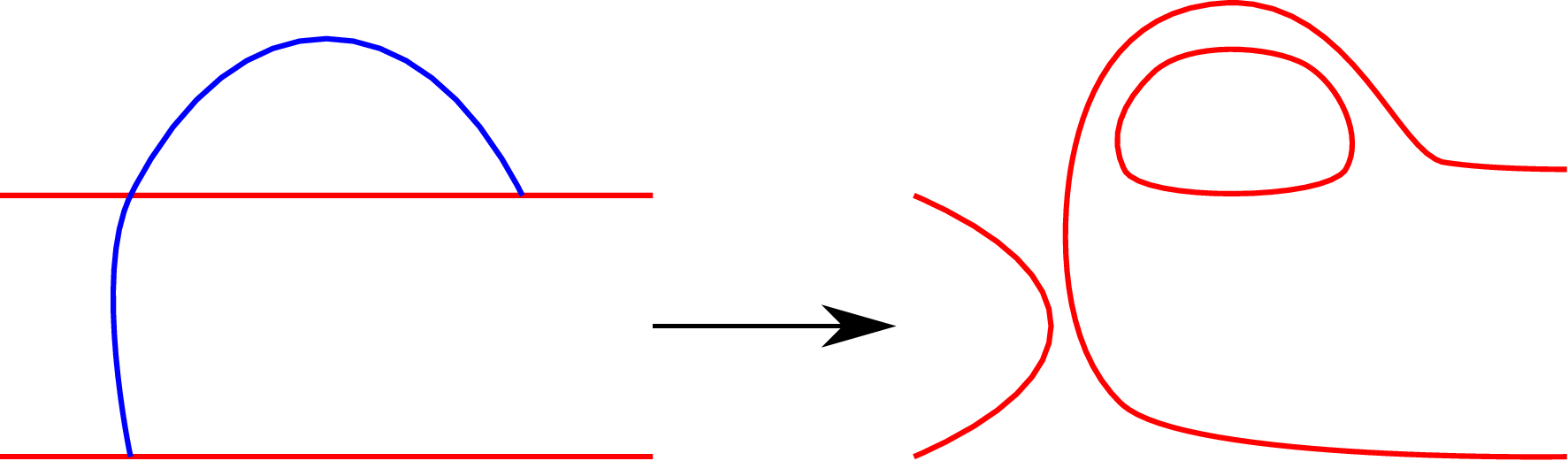}
	\caption{A disallowed bypass attachment results in an overtwisted \linebreak contact structure.}
	\label{fig:disallowedBP}
\end{figure}

There are some general ways which help us find bypasses in many situations.
A component of the dividing set on a convex surface with Legendrian
boundary is called \emph{boundary-parallel} if it cuts off a half-disk
containing no other components of the dividing set. We will use the
standard fact that, except for a disk with
$\tb(\partial\Sigma)=-1$, a boundary-parallel dividing arc gives a
bypass \cite[Proposition 3.17 and 3.18]{HondaOnTheI}. Consequently, such a bypass may be used
to destabilize the corresponding Legendrian boundary component.

Our last tool mentioned here for finding bypasses is a proposition by Honda, Kazez and Mati\'c \cite{HKMrotation} which enables us to swap a bypass to another one which may be more convenient to examine. To state this, we need to introduce a notion:

Let $\Sigma$ be a convex boundary component of a tight 3-manifold $(Y, \xi)$ and let $\Gamma$ denote the dividing set on $\Sigma$. Consider $\Sigma$ with the opposite orientation as the one induced from the orientation on $Y$. Let $\alpha$ and $\alpha'$ be two disjoint Legendrian arcs in $\Sigma$ that start and end on $\Gamma$, start on an arc $\gamma_1\subset \Gamma$, and both intersect $\Gamma$ in exactly three points. Suppose that there is an embedded rectangle $R$ in $\Sigma$ such that its two sides are subarcs of $\alpha$ and $\alpha'$ and its other two sides are subarcs of $\gamma_1, \gamma_2\subset\Gamma$, see Figure \ref{fig:bypass_rotation} for an illustration. Assume furthermore, that $\alpha$ and $\alpha'$ extend beyond $\gamma_2$, not reintersecting $\partial R$, and that with the orientation induced from $R$, the curve $\gamma_2$ starts on $\alpha'$ and ends on $\alpha$. Then we say that $\alpha$ is \emph{to the left} of $\alpha'$.

\begin{figure}[h!]
	\centering
	\includegraphics[width=0.33\textwidth]{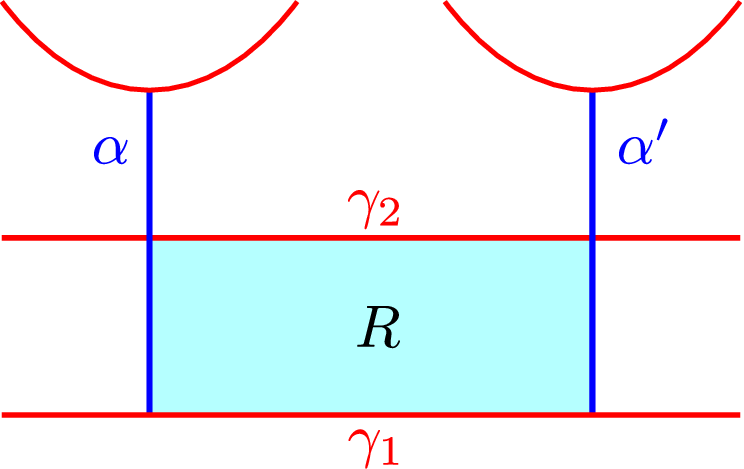}
	\caption{$\alpha$ is to the left of $\alpha'$. The figure is local, any of the four dividing arcs may be part of the same dividing curve in $\Gamma$. If there is a bypass for $\alpha'$ then there is a bypass also for $\alpha$.}
	\label{fig:bypass_rotation}
\end{figure}\vspace{-0.2cm}

\begin{proposition}[Bypass Rotation, Honda-Kazez-Mati\'c, \cite{HKMrotation}]\label{prop:bypass_rotation}
	Let $(Y, \xi)$ be a tight contact 3-manifold and $\Sigma\subset (Y,\xi)$ a convex surface. Let $\alpha$ and $\alpha'$ be two disjoint Legendrian arcs in $\Sigma$ such that $\alpha$ is to the left of $\alpha'$. If there is a bypass for $\Sigma$ attached along $\alpha'$ from the front side of the diagram, then there is a bypass for $\Sigma$ attached along $\alpha$ from the front.
\end{proposition}
For the proof see  \cite[Lemma 4.2]{HKMrotation}.

Through standard neighborhoods we obtain a useful correspondence between bypasses and destabilizations, see also \cite[Lemma 3.15 and 4.4]{HondaOnTheI}, \cite[Section 2.4.]{chakraborty2020cabling} and \cite{EtnyreHonda01}.

\begin{lemma}[Twist Number Lemma, Honda, \cite{HondaOnTheI}]
	\label{lem:twist_number}
	Let $(Y,\xi)$ be a tight contact $3$-manifold and let
	$L\subset(Y,\xi)$ be a Legendrian knot with $\tb(L)=n$.
	Denote the standard contact neighborhood of $L$ by $N(L)$.
	If there exists a bypass attached from outside $N(L)$ along
	a Legendrian ruling curve of slope $r$ on $\partial N(L)$
	and $\frac{1}{r}>n+1$, then there is a Legendrian knot $L'$
	smoothly isotopic to $L$ with larger $\tb$.
	The convex torus obtained by attaching the bypass bounds
	an enlarged standard contact neighborhood of $L'$, and
	$L'$ is a destabilization of $L$.
\end{lemma}

Under the assumptions of Lemma \ref{lem:twist_number},
attaching the bypass corresponds to a destabilization of $L$
and increases its Thurston--Bennequin number.

\section{The upper bound}\label{sec:upper_bound}

We now prove the main results stated in the Introduction. Although
Theorem \ref{thm:main} and Proposition \ref{thm:max_tb} concern the family
$K(4,m)$, some of the intermediate results in Sections
\ref{subsec:on_the_sphere} and \ref{subsec:inside} hold for general
double twist knots $K(k,m)$.

We begin by determining the maximal Thurston--Bennequin invariant of
$K(4,m)$.

\begin{proof}[Proof of Proposition \ref{thm:max_tb}]
	The proof is an application of Proposition \ref{prop:ruling_maxtb} stating that the existence of a normal ruling implies that the $\tb$ is maximal. In each case we give a front diagram of a Legendrian knot realizing the knot type and show that it admits a normal ruling. In Figure \ref{fig:max_tbk_ruling}, we present such diagrams in a way that the number of crossings in the box is $|m+2|$ due to the two crossings on the left and right hand side. By Proposition \ref{prop:ruling_maxtb}, this is enough to prove that the given knots maximize $\tb$. The cases $m=-1$ and $m=0$ are treated separately below.
	
	\begin{figure}[h!]
		\centering
		\includegraphics[width=0.88\textwidth]{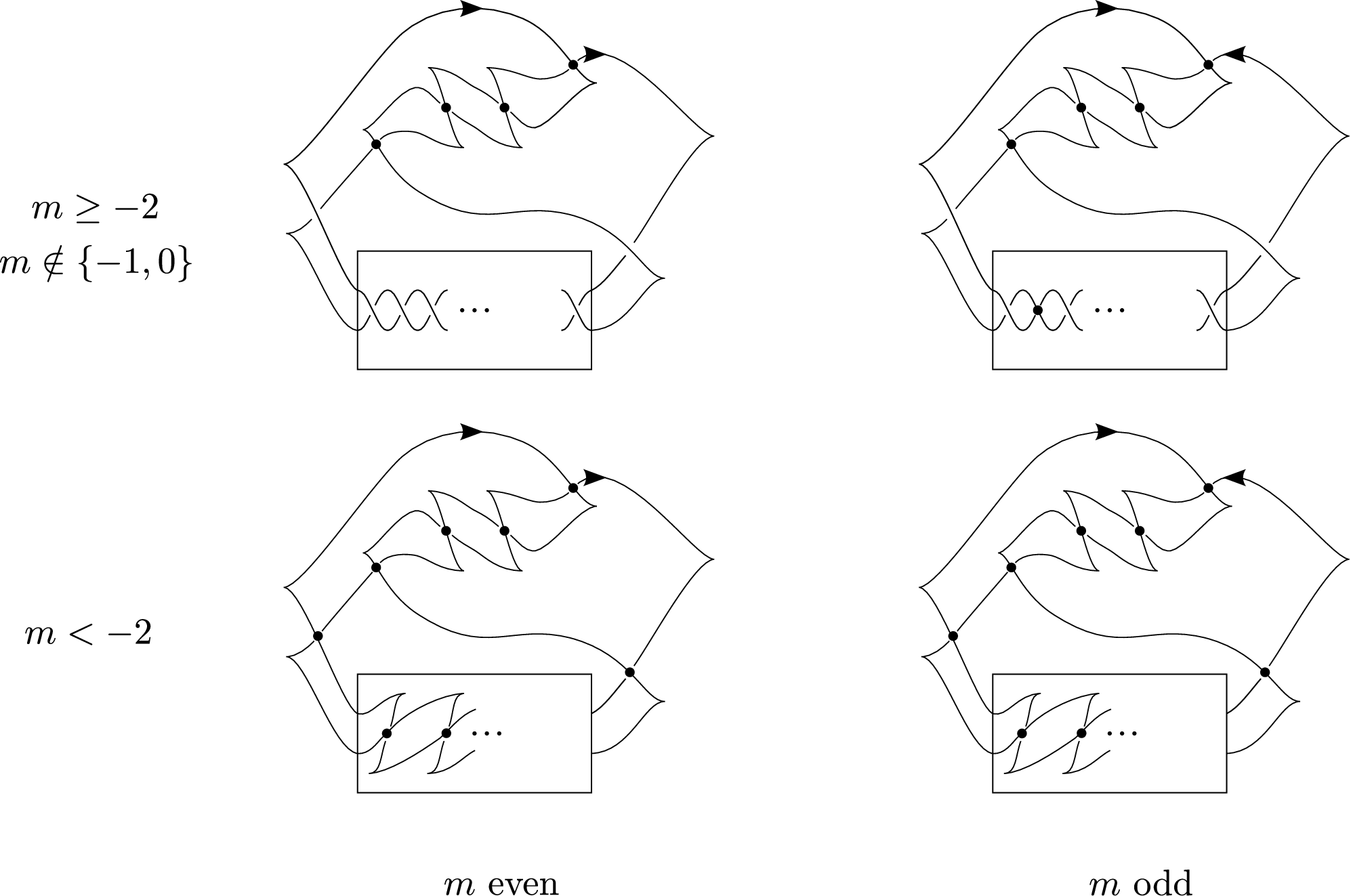}
		\caption{Maximal $\tb$ front diagrams of the knot type
				$K(4,m)$ for $m\neq-1,0$. The black dots indicate the
				switches of normal rulings.}
		\label{fig:max_tbk_ruling}
	\end{figure}
	
	Now our claim follows by easy computation. Recall from Proposition \ref{prop:classical_from_front} that $\tb=\wr-\frac{1}{2}\cdot c$. We need to observe how the orientation of the strands and the number of cusps depends on $m$, see Figure \ref{fig:max_tbk_ruling}.
		\begin{figure}[h!]
			\centering
			\includegraphics[width=0.4\textwidth]{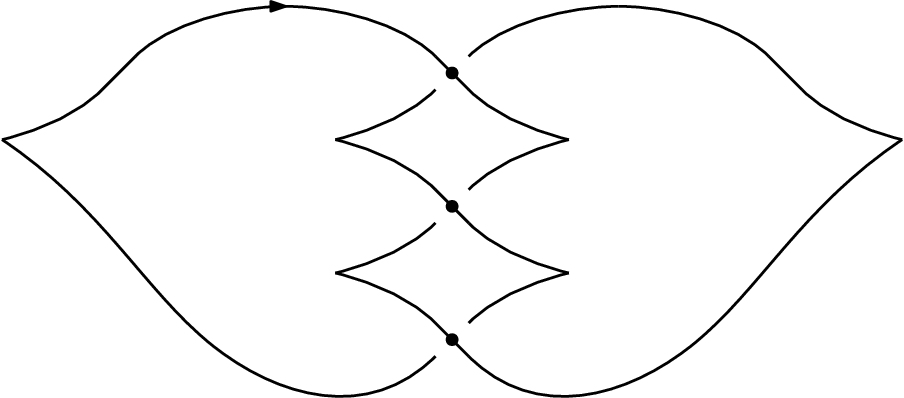}
			\caption{A maximal $\tb$ front diagram of the negative
				trefoil. The black dots indicate the switches of a
				normal ruling.}
			\label{fig:max_tb_trefoil}
		\end{figure}
	\begin{itemize}
		\item If $m\geq-2$ is even and $m\neq0$, the crossings in the box are negative, all the others are positive, and cusps only occur outside the box. Therefore $$\tb=6-(m+2)-\frac{1}{2}\cdot 10=-1-m.$$
	\item For $m=-1$, the knot is the negative trefoil.
	We use the diagram in Figure \ref{fig:max_tb_trefoil},
	which admits a normal ruling and therefore has maximal $\tb$.
	The three crossings are negative and there are six cusps.
	Therefore $$\tb=-3-\frac{1}{2}\cdot6=-6.$$
	\item For $m=0$, the knot is the unknot. We use its standard
	front diagram with two cusps and no crossings, which admits
	a normal ruling consisting of a single ruling disk.
	By Proposition \ref{prop:ruling_maxtb}, it has maximal $\tb$.
	Therefore $$\tb=0-\frac{1}{2}\cdot2=-1.$$
		\item If $m>-1$ and odd, the crossings in the box are negative, outside the box there are $2$ positive and $4$ negative crossings and cusps only occur outside the box.
		Therefore $$\tb=-2-(m+2)-\frac{1}{2}\cdot 10=-9-m.$$
		\item If $m<-2$ and even, all crossings are positive and near each crossing inside the box there are two cusps, hence the contribution of the braid in the box is $0$.
		Therefore $$\tb=6-\frac{1}{2}\cdot 10+0=1.$$
		\item If $m<-2$ and odd, there are $2$ positive and $4$ negative crossings outside the box, inside all crossings are positive and to each belong $2$ cusps, hence again, the contribution of the braid in the box is $0$.
		Therefore $$\tb=-2-\frac{1}{2}\cdot 10+0=-7.$$
	\end{itemize}
\end{proof}

\begin{remark}
	The same construction applies to every $k>2$ for which
	$K(k,m)$ is a knot. For $m\neq-1,0$, the fronts have
	$k+2$ crossings and $2k+2$ cusps outside the box.
	The normal rulings extend by repeating the local pattern
	in the upper chain.
	
	For even $k$, the formulas in Proposition \ref{thm:max_tb}
	remain valid with $-9-m$, $-7$ and $-6$ replaced by
	$-2k-1-m$, $1-2k$ and $-2(k-1)$, respectively.
	For $m=-1$, we use the standard front of the negative
	$(2,k-1)$-torus knot, with $k-1$ negative crossings,
	$2(k-1)$ cusps and a normal ruling.
	
	For odd $k$, necessarily $m$ is even. The $k$ crossings
	in the upper chain are positive and the other two crossings
	outside the box are negative, so the contribution outside
	the box is $-3$. The box contributes $m+2$ if $m\geq2$
	and $2(m+2)$ if $m\leq-2$. Together with the unknot case
	$m=0$, this gives
	$$
	\tb_{\max}(K(k,m))=
	\begin{cases}
		m-1, & \textrm{if } m\geq0,\\
		2m+1, & \textrm{if } m<0.
	\end{cases}
	$$
\end{remark}

We now turn to the proof of Theorem \ref{thm:main}. Some of the
intermediate results below are proved for general double twist knots
$K(k,m)$. As outlined in the Introduction, we cut $S^3$ and the knot
along a $2$-sphere separating the two twist regions. In Subsection
\ref{subsec:on_the_sphere}, we make this sphere convex and normalize
its dividing set. In Subsection \ref{subsec:inside}, we bound the possible inner Legendrian pieces. In Subsection \ref{subsec:outside}, we bound
the possible contact structures on the outer complement.
Combining these bounds gives Theorem \ref{thm:main}.

\subsection{The dividing curve on the sphere}\label{subsec:on_the_sphere}

Our aim in the following sections is to prove Theorem \ref{thm:main}. First of all, we introduce some notations and make some preliminary observations. 

In the next sections it will be particularly important to distinguish between Legendrian and smooth knots. Therefore, recall that we denote Legendrian knots by calligraphic capital letters, while smooth knots by regular capital letters. Consider a double twist knot $K(k,m)\subset S^3$. Observe, that $k$ or $m$ has to be even, otherwise (when $k$ and $m$ are both odd numbers) $K(k,m)$ is a 2-component link.

Recall that the cases $|m|<2$ are already known, as discussed in the introduction.
Moreover, if $|m|=2$, then $K(k,m)$ is a twist knot, which is completely classified by Etnyre, Ng and V\'ertesi \cite{ENV}. On the other hand, our argument also applies to the case $m=-2$, so we include this case.
Thus, in what follows, we assume that $k>2$ and either $|m|>2$ or
$m=-2$.

Let $\K$ denote the smooth knot type of $K(k,m)$. There is an embedded 2-sphere $S$ in $S^3$ which intersects $\K$ in four points and divides $S^3$ into two 3-balls $\Bout$ and $\Bin$ such that $\K \cap \Bout$ is a vertical 2-braid with $k$ half-twists, while $\K\cap \Bin$ is a horizontal 2-braid with $m$ half-twists according to Figure \ref{fig:K(k,m)_and_sphere}.

\begin{figure}[h!]
	\centering
	\includegraphics[width=0.45\textwidth]{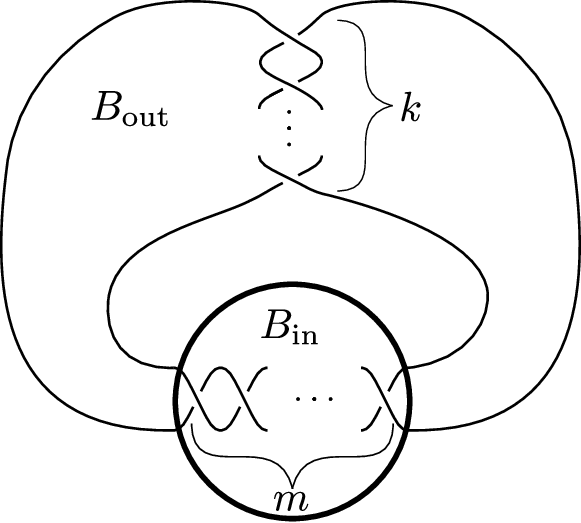}
	\caption{The sphere $S$ divides $S^3$ into two 3-balls $\Bin$ and $\Bout$. The intersection of $\K$ with $\Bout$ is a vertical 2-braid with $k$ half-twists, while its intersection with $\Bin$ is a horizontal 2-braid with $m$ half-twists.}
	\label{fig:K(k,m)_and_sphere}
\end{figure}

 Fix a standard model of $\K$ as shown in Figure \ref{fig:K(k,m)_and_sphere}, and take a Legendrian realization $\L$ of $K(k,m)$. As $\L$ is smoothly isotopic to $K(k,m)$, there is an isotopy $\psi: S^3\rightarrow S^3$ taking $K(k,m)$ to $\L$ so that the image $\psi(S)$ of the sphere $S$ is a convex surface. According to Theorems \ref{thm:standard_neighborhood} and \ref{thm:LegendrianRealizationPrinciple}, we can choose a sufficiently small standard neighborhood $N$ of $\L$ such that its boundary intersects $\psi(S)$ in four Legendrian unknots.
First, we normalize the pull-back $\psi^{-1}(\Gamma)$ of the dividing set on $S$. Then, in Sections \ref{subsec:inside} and \ref{subsec:outside}, we will examine $\K\cap \Bin$ and $\K\cap \Bout$. The results presented here are based on generalizations of the proofs in \cite{ENV}.

\begin{theorem}\label{thm:normalizingGamma}
Let $k>2$ and assume that either $|m|>2$ or $m=-2$, and fix the double twist knot $K(k,m)$ and the sphere $S$ as above. Consider a tubular neighborhood $\nu (K(k,m))$ of $K(k,m)$ and let $P=S\setminus \nu(K(k,m))$, the sphere with four disks removed. For any Legendrian realization $\L$ of $K(k,m)$, there is an isotopy $\psi: S^3\rightarrow S^3$ such that $\psi(S)$ and $\psi(P)$ are convex surfaces, $\psi(\partial P)$ are four Legendrian unknots with $\tb=-1$, $N=\psi(\nu (K(k,m)))$ is a standard contact neighborhood of $\L$ and the pullback $\Gamma_P \subset P$ of the dividing curves on $\psi(P)$ is normalized as in Figure \ref{fig:dividings_on_sphere}.
\end{theorem}
\begin{figure}[h!]
	\centering
	\includegraphics[width=0.5\textwidth]{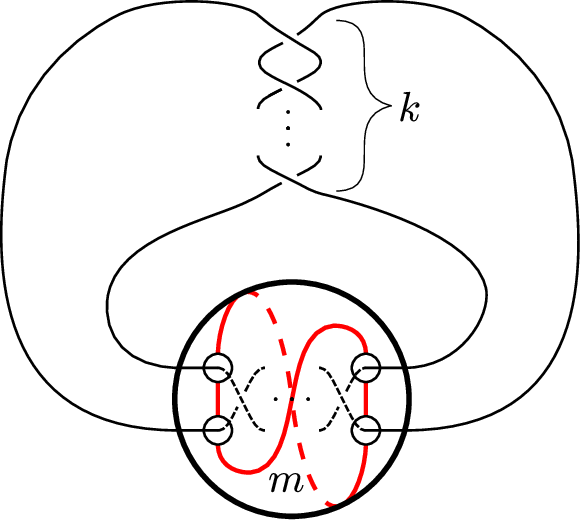}
	\caption{The normalization of $\Gamma_P$ on $\psi(P)$}
	\label{fig:dividings_on_sphere}
\end{figure}

For the sake of easier notations, we will identify every set with its image by $\psi$ and use the same notations throughout the proofs and lemmas.  

To verify Theorem \ref{thm:normalizingGamma}, we need the following lemma.

\begin{lemma}\label{lem:Unknot}
Let $k>2$ and assume that either $|m|>2$ or $m=-2$. Consider a Legendrian realization $\L$ of the double twist knot $K(k,m)$. There exists a Legendrian unknot $\U$ such that $\tb(\U)=-1$ and $\U$ is topologically isotopic to the curve $B$ of Figure \ref{fig:Unknot_B} in the complement of $\L$.
\end{lemma}
\begin{figure}[h!]
	\centering
	\includegraphics[width=0.5\textwidth]{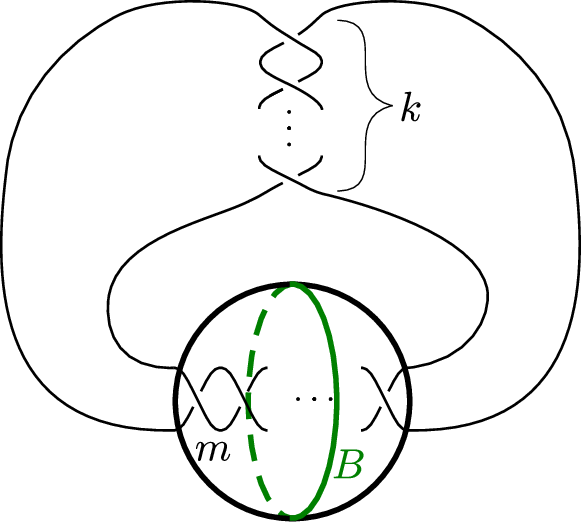}
	\caption{The curve $B$ in the complement of $\L$. There is a Legendrian unknot $\U$ in the complement of $\L$ and in the topological class of $B$ with $\tb(\U)=-1$.}
	\label{fig:Unknot_B}
\end{figure}

We will prove Lemma \ref{lem:Unknot} later in this section, because we first need some further notations and statements. However, to motivate the following discussion, we give a short idea of the proof here: In the topological class of $B$, we can always find a Legendrian knot $\U$ smoothly isotopic to the unknot and disjoint from $\L$. Thus, it is enough to concentrate on $\tb(\U)$. In the standard contact structure, $\tb(\U)=-n$ for some $n>0$. Suppose that $\U$ is chosen in the complement of $\L$ so that $\tb(\U)$ is maximal among Legendrian knots in its topological class. Our aim will be to show that $n>1$ leads to a contradiction.

According to the above reason, in what follows, we suppose that $\U$ is a Legendrian knot in the topological class of $B$, which is smoothly isotopic to  the unknot and disjoint from $\L$. Moreover, we assume that  $\tb(\U)=-n$ for some $n>0$ is maximal. Next, we introduce some notations.

Consider a standard contact neighborhood $N_{\U}$ of $\U$ that is disjoint from $\L$, and let $Q=\overline{S^3\setminus N_{\U}}$ a solid torus $S^1\times D^2$ with convex boundary.
On $\partial Q$, we use the meridian and longitude of $Q$,
corresponding to the longitude and meridian of $N_\U$,
respectively. In these coordinates, the two dividing curves
have slope $-n$.

By Lemma \ref{lemma:ruling_barmi}, ruling curves on $Q$ can be assumed to be meridional, and we can choose two disks $D_1$ and $D_2$ in $Q$ bounded by ruling curves, as shown in \mbox{Figure \ref{fig:disks_D12}}.

\begin{figure}[h!]
	\centering
	\includegraphics[width=0.45\textwidth]{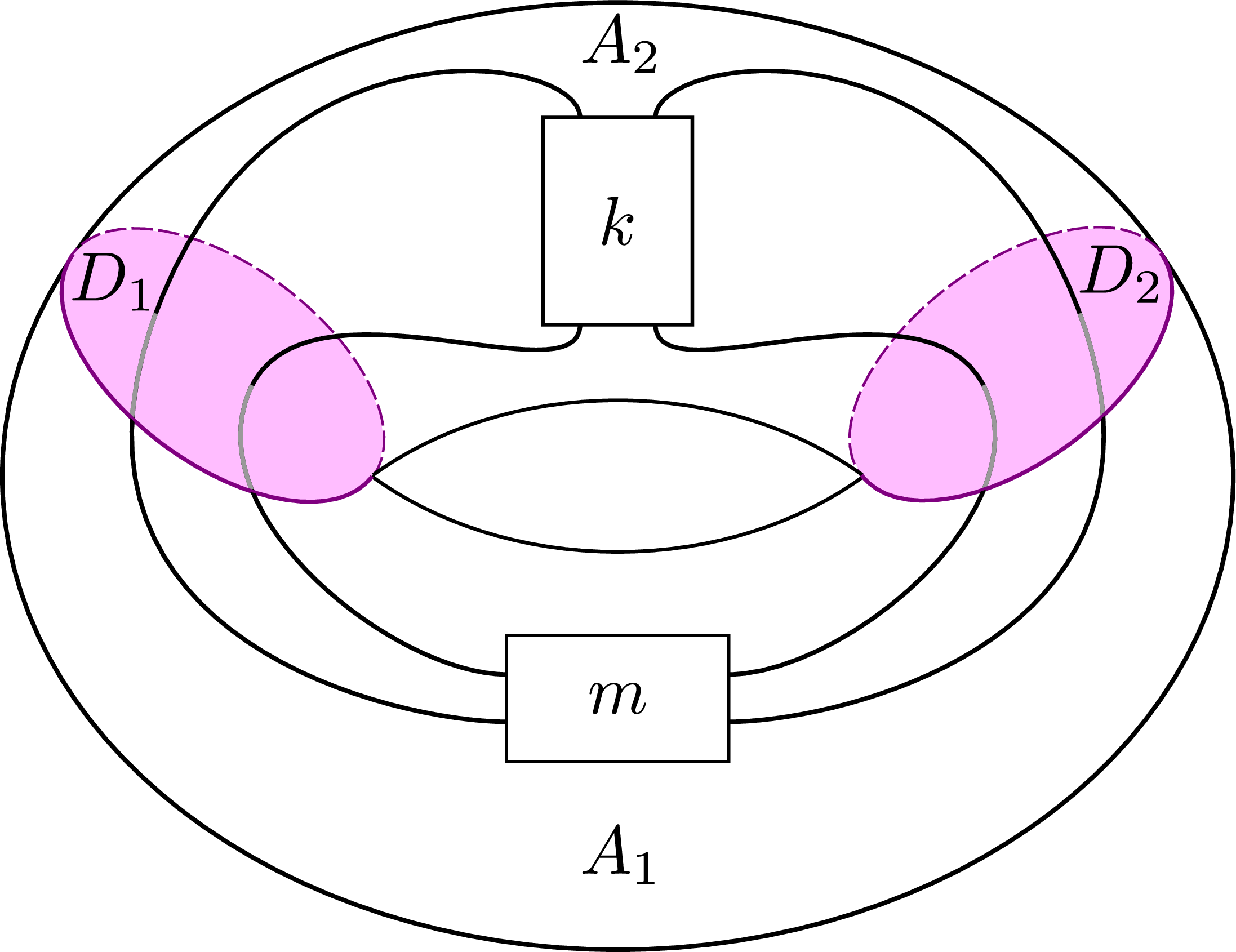}
	\caption{Two disks $D_1$ and $D_2$ in $Q$ bounded by meridional ruling curves}
	\label{fig:disks_D12}
\end{figure}

Isotope $D_1$ and $D_2$ so that they become convex, and that they both intersect the neighborhood $N$ of $\L$ in two disks with Legendrian boundary of $\tb=-1$.
Now $\partial Q\setminus (\partial D_1 \cup \partial D_2)$ are two annuli $A_1$ and $A_2$ and after rounding corners, we can think of $A_1\cup D_1\cup D_2$ as the sphere $S$.

For $i=1,2$, let $P_i=D_i\setminus N$ a pair of pants with three boundary components $c_{i,1}$, $c_{i,2}$ and $c_{i,3}$, where $c_{i,3}$ is the boundary component in $\partial Q$ and by Lemma \ref{lemma:ruling_barmi}, we can assume that $c_{i,1}$ and $c_{i,2}$ are ruling curves in $\partial N$.

Let $A_{i,j}$ denote an annular neighborhood of $c_{i,j}$ in $P_i$, and define $P_i'$ as the closure of the complement of the $A_{i,j}$ regions in $P_i$, that is,
$$P_i'=\overline{P_i\setminus \bigcup\limits_{j=1}^3 A_{i,j}}.\vspace{-0.2cm}$$

\begin{figure}[h!]
	\centering
	\includegraphics[width=0.4\textwidth]{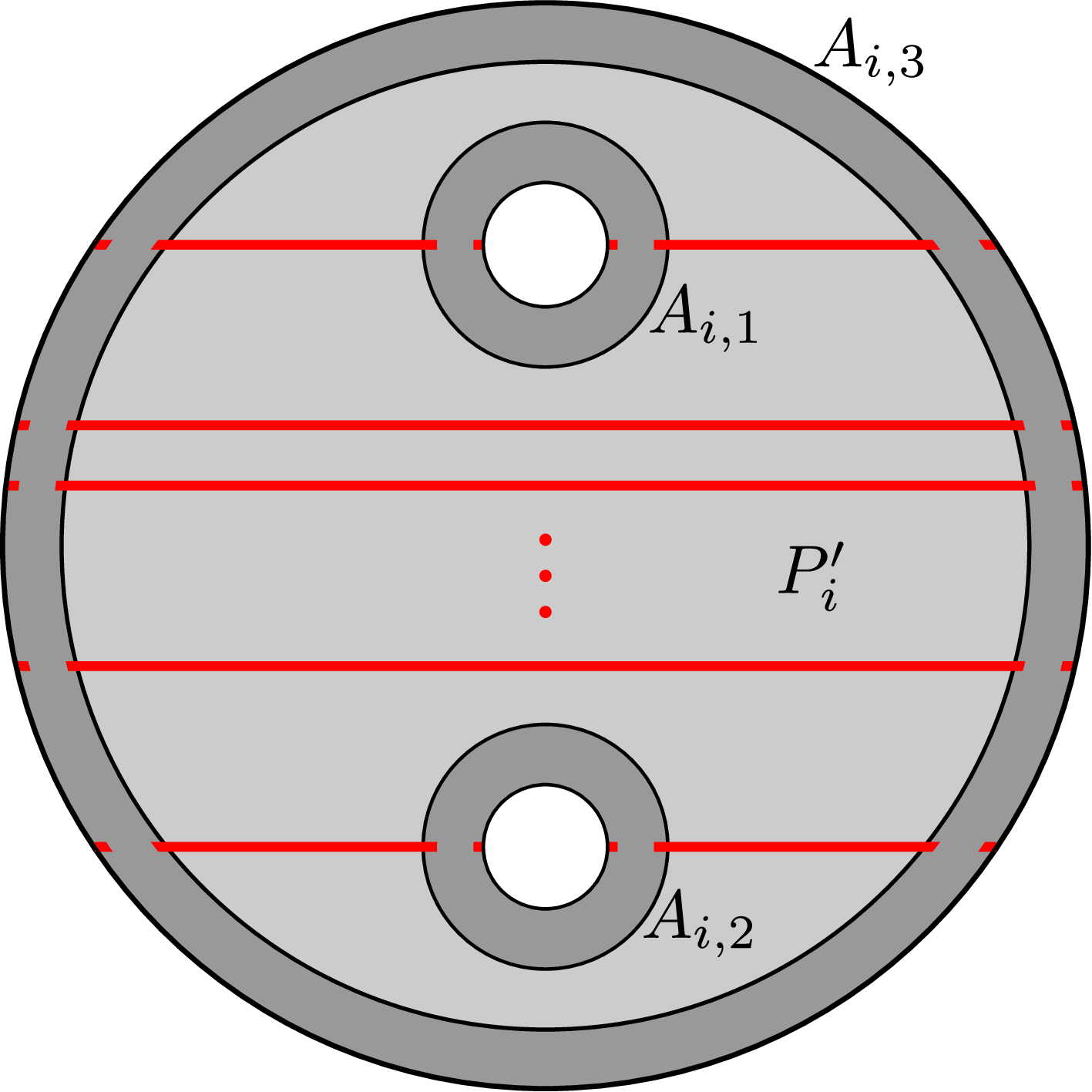}
	\caption{The dividing curves on $P_i$. The inner circles are $c_{i,1}$ and $c_{i,2}$, while $c_{i,3}$ is the outer boundary component of $P_i$. The dark annuli are $A_{i,j}$ and $P_i'$ is the lighter shaded region. The dividing curves are horizontal segments in $P_i'$ and they twist in $A_{i,j}$ in a non-intersecting way.}
	\label{fig:Fig19ENV}
\end{figure}

\begin{lemma}\label{lem:GammaPi}
	Assume $n>1$. The dividing curves on $P_i$ are as in Figure \ref{fig:Fig19ENV}. More precisely, on $P_i'$, we can imagine the dividing curves as horizontal segments intersecting the boundary component of $A_{i,1}$ and $A_{i,2}$ in 2, while the boundary component of $A_{i,3}$ in $2n$ points. On $A_{i,j}$, the dividing curves can be assumed to be the (non-intersecting) extensions of the dividing curves on $P_i'$ with some number of twists between the above mentioned intersection points.
\end{lemma}

\begin{proof}
Observe that the dividing curves run around $\partial Q$ once in the meridional and $n$ times in the longitudinal direction. Therefore they intersect $c_{i,3}$ $2n$ times. Furthermore, the dividing curves on $N$ intersect $c_{i,1}$ and $c_{i,2}$ twice. This means that along each inner boundary component of $P_i$ ($c_{i,1}$ or $c_{i,2}$) there can be at most one boundary parallel dividing curve. Thus, if there are more than two boundary parallel dividing curve components on $D_i$, then at least one of them is also boundary parallel in $P_i$, with the boundary component $c_{i,3}$. Having a boundary parallel dividing curve component on $c_{i,3}$ would allow us to do a bypass and destabilize $\U$ in the complement of $\L$, which would contradict the maximality of $\tb(\U)$. So we know that the dividing curves on $P_i$ are as in Figure \ref{fig:Fig19ENV}.
\end{proof}

We continue our preparations with some further definitions and observations. Recall that we thought of $A_1\cup D_1 \cup D_2$ as the sphere $S$. Now we define a curve $\gamma$ in $S$ for $\L$ such that it bounds a disk $\Dout$ in $\Bout$, disjoint from $\L$. Although we generally assume $k>2$ in this section, we describe
the following geometric construction for both signs of $k$.
Depending on the number $k$ in $K(k,m)$, let
$\gamma_0\subset S$ be the curve shown in Figure
\ref{fig:gammacurve}, and define $\gamma$ as the image
of $\gamma_0$ after $\left\lfloor\frac{|k|}{2}\right\rfloor$
Dehn twists on the sphere $S$ along the curve $t$ shown
in Figure \ref{fig:tcurve} (positive Dehn twists if $k$
is positive and negative ones if $k$ is negative). For an example, see Figure \ref{fig:Plgammadefk2} illustrating the definition of $\gamma$ in the $k=2$ case.

\begin{figure}[ht]
	\centering
	\includegraphics[width=0.6\textwidth]{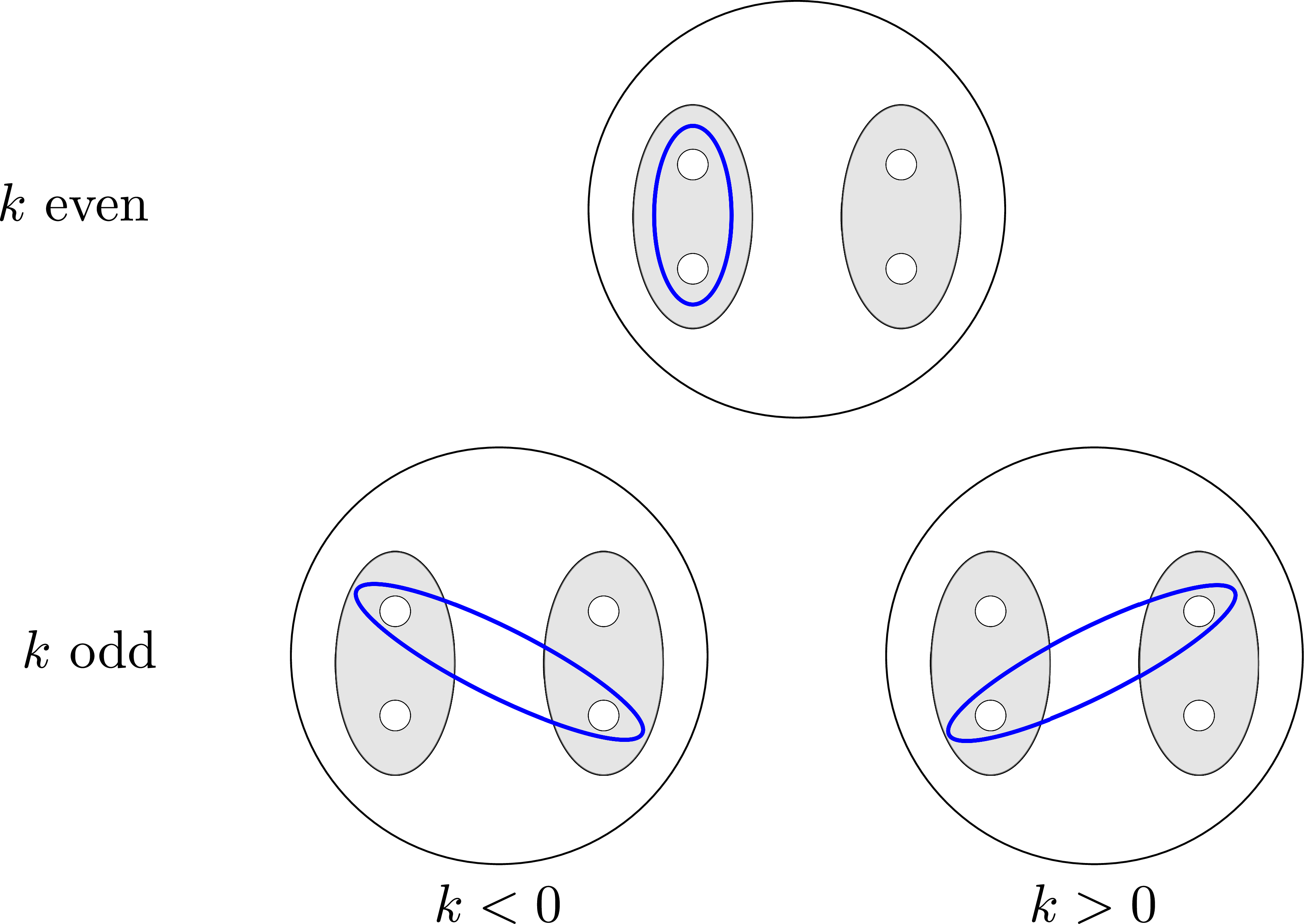}
	\caption{The curve $\gamma_0$ on $S$. The shaded regions are $P_1$ and $P_2$. Note that the definition of $\gamma_0$ depends on the number $k$ in $K(k,m)$.}
	\label{fig:gammacurve}
\end{figure}

\begin{figure}[ht]
	\centering
	\includegraphics[width=0.43\textwidth]{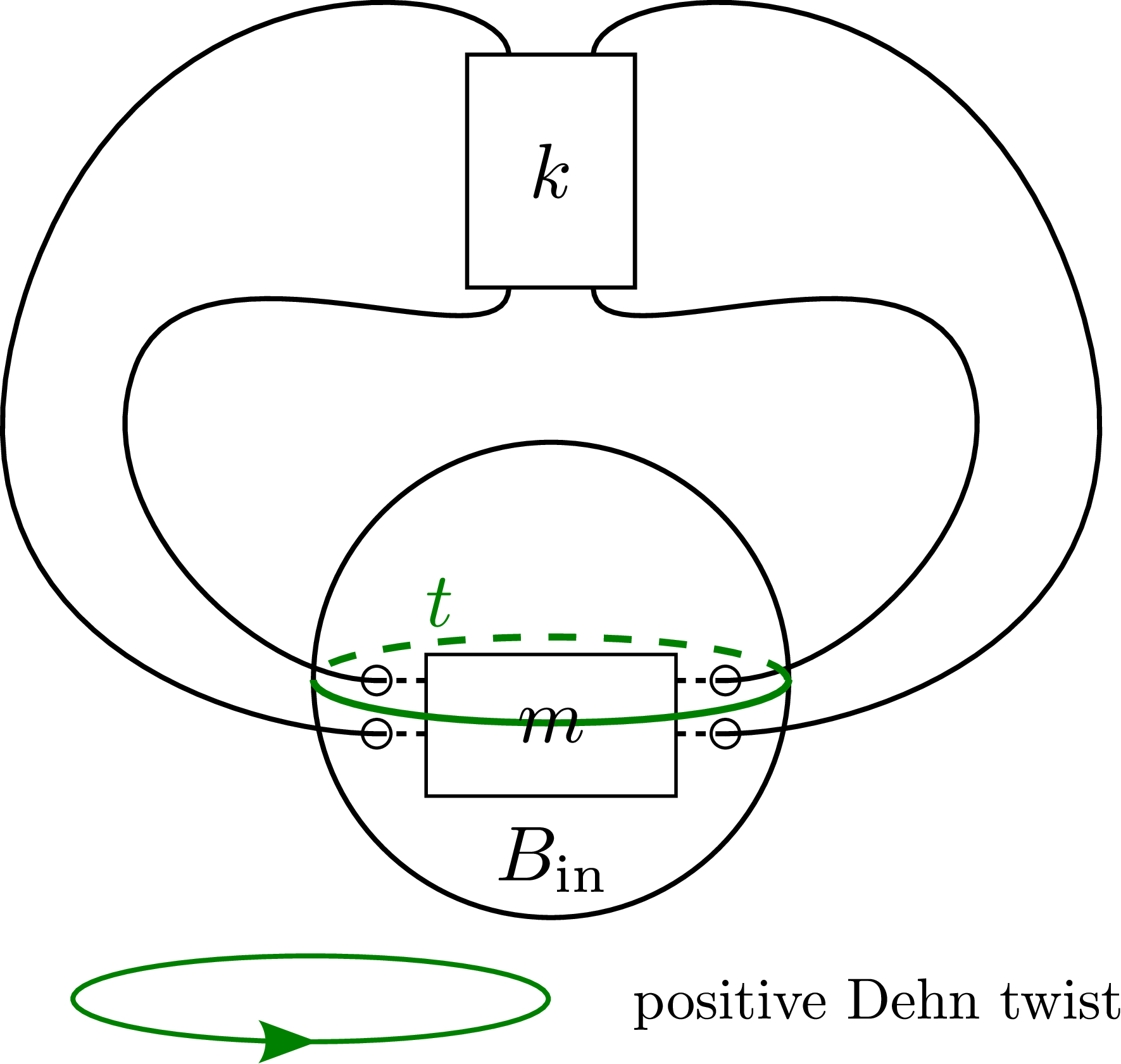}
	\caption{The curve $t$ in $S$ along which we do Dehn twists on $\gamma_0$ to obtain $\gamma$.}
	\label{fig:tcurve}
\end{figure}

\begin{figure}[h!]
	\centering
	\includegraphics[width=0.7\textwidth]{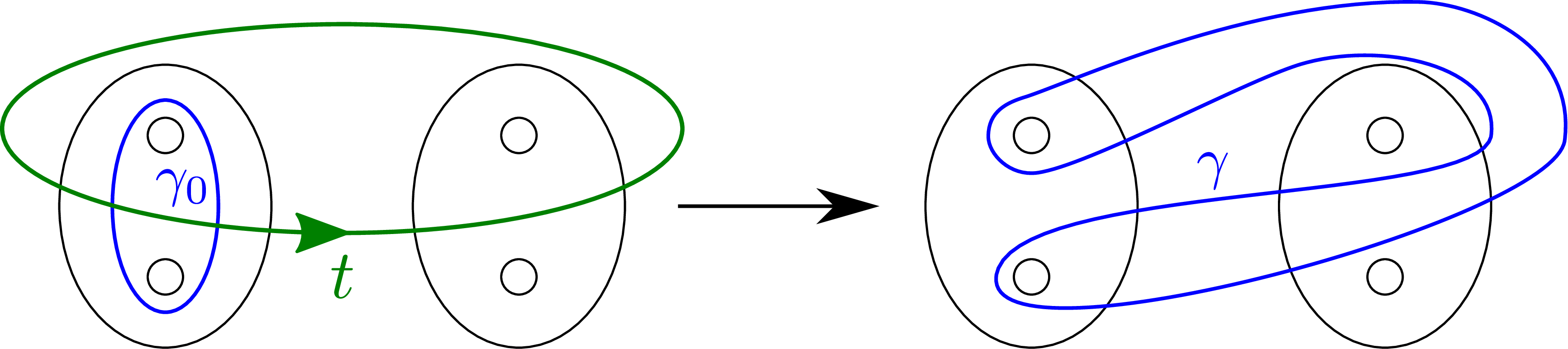}
	\caption{If $k=2$, we obtain the curve $\gamma$ from $\gamma_0$ by a positive Dehn twist along $t$ on $S$.}
	\label{fig:Plgammadefk2}
\end{figure}

For even $k$, first consider an unknot in $S^3$ obtained by fixing $K(k,m)\cap \Bin$ and instead of completing this part with $K(k,m)\cap \Bout$, close it with two unknotted arcs in $\Bout$ as shown on the left hand side of Figure \ref{fig:K(k,m)_lezarasa}. 

\begin{figure}[h!]
	\centering
	\includegraphics[width=0.9\textwidth]{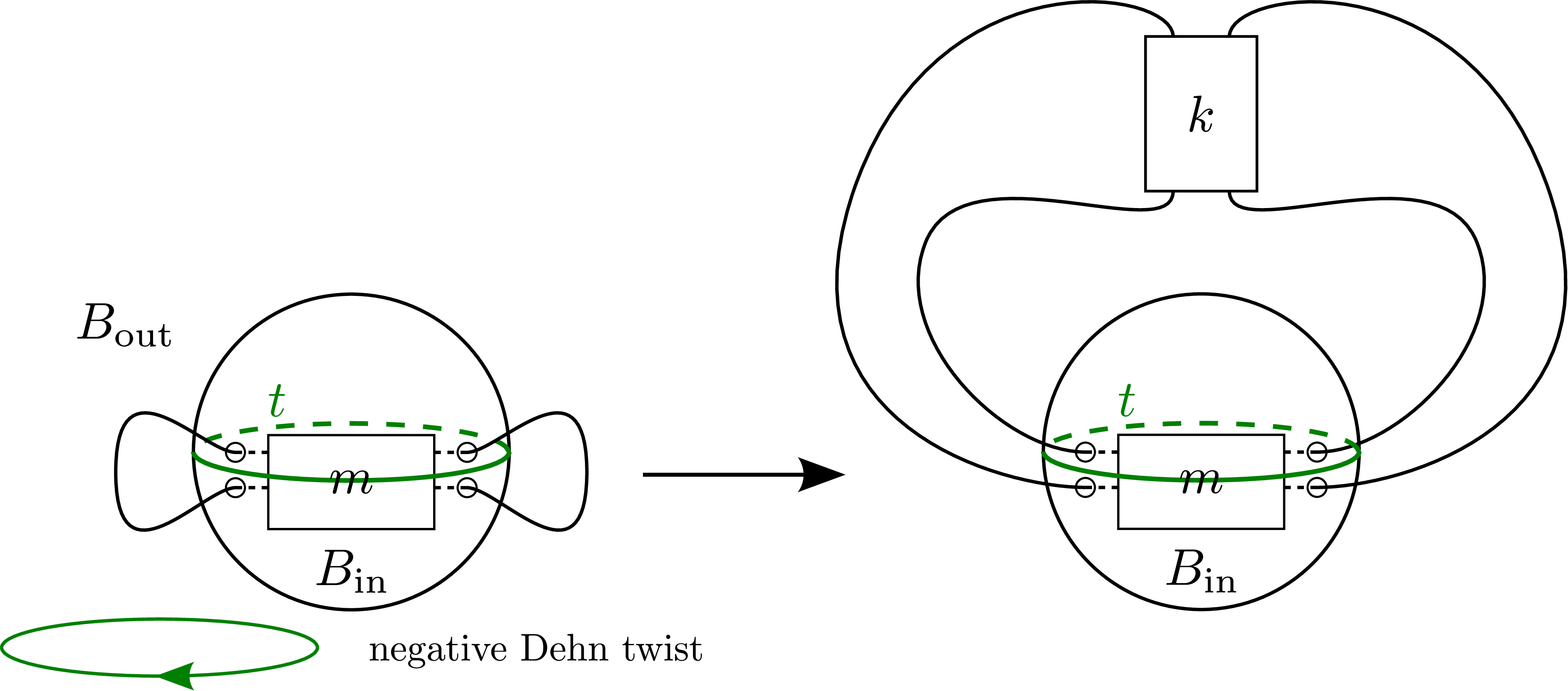}
	\caption{For even $k$, closing $K(k,m)\cap\Bin$ with two
		unknotted arcs in $\Bout$. From this unknot we obtain
		$K(k,m)$ by doing $k/2$ full twists in $\Bout$.}
	\label{fig:K(k,m)_lezarasa}
\end{figure}\vspace{-0.2cm}

Note that from this unknot, we can get $K(k,m)$ by
twisting the upper hemisphere over $t$ and extending
it onto $\Bout$, as on the right hand side of Figure
\ref{fig:K(k,m)_lezarasa}. For $k>0$ we need
$\frac{k}{2}$ negative full twists, while for $k<0$
we need $\frac{|k|}{2}$ positive ones.
With the boundary convention used in Figure
\ref{fig:Plgammadefk2}, these correspond to positive,
respectively negative, Dehn twists along $t$ on $S$.
In the untwisted model, $\gamma_0$ bounds a disk in
$\Bout$ disjoint from the closing arcs. Transporting
the arcs and the disk together gives a disk bounded
by $\gamma$ which is still disjoint from $\L$.

For odd $k$, start with the one-half-twist outer
tangle corresponding to $k=1$ or $k=-1$, according
to the sign of $k$, and the corresponding curve
$\gamma_0$ in Figure \ref{fig:gammacurve}.
This curve bounds a disk disjoint from the tangle.
The remaining $\frac{|k|-1}{2}$ full twists, negative
if $k>0$ and positive if $k<0$, transport this disk
in the same way, see Figure \ref{fig:gammaboundsdisk}
for an illustration.

 \begin{figure}[h!]
 	\centering
 	\includegraphics[width=0.4\textwidth]{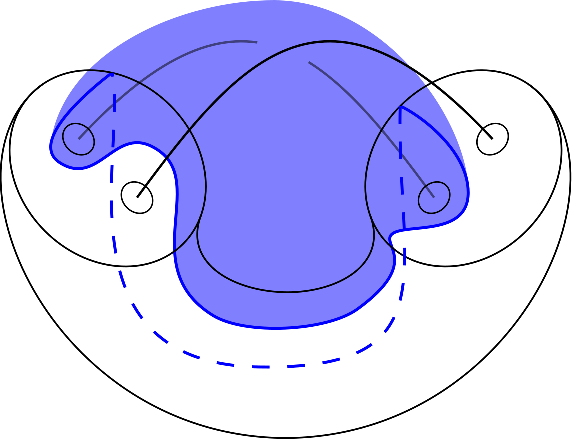}
 	\caption{The curve $\gamma$ bounds a disk in $\Bout$
 		disjoint from $\L$. This figure illustrates the case
 		$k=-1$, where $\gamma=\gamma_0$, and the disk is the
 		blue shaded region. The $k=1$ model is obtained by
 		reflecting the picture.}
 	\label{fig:gammaboundsdisk}
 \end{figure}

Recall the notations of Figure \ref{fig:Fig19ENV}, and consider the annulus $A=A_1\cup A_{1,3}\cup A_{2,3}$ on the sphere, see Figure \ref{fig:annulusA}. We can assume that all the intersection points of the curve $\gamma$ and the dividing curve of $S$ are in $A$. Examine the intersection of $\gamma$ with the annulus $A$. We can choose $\gamma_0$ so that it intersects $A$ in 2 arcs if $k$ is odd, while is disjoint from $A$ if $k$ is even (as in Figure \ref{fig:gammacurve}). After each Dehn twist 4 further arcs of intersection occur (2 at the front and 2 at the back of the sphere $S$), see \mbox{Figure \ref{fig:slopes_pspt}}. Therefore, in every case, the intersection of $\gamma$ with the annulus $A$ consists of $4\cdot \frac{k}{2}=2+4\cdot\frac{k-1}{2}=2k$ arcs, let us denote them by $\gamma_1,\gamma_2,...,\gamma_{2k}$.

\begin{figure}[h!]
	\centering
	\includegraphics[width=0.4\textwidth]{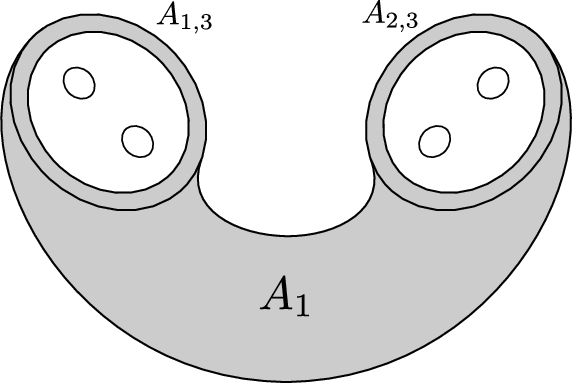}
	\caption{The annulus $A=A_1\cup A_{1,3}\cup A_{2,3}$ on the sphere $S$}
	\label{fig:annulusA}
\end{figure}

Isotope $\gamma$ so that it intersects $P_i'$ in horizontal arcs, and identify the closure of $A$ with $S^1\times [0,1]$. We understood the dividing curves on $\partial Q$. Consequently, we also know that there are $2n$ arcs of the dividing curves on $A$, running from one boundary component to the other in a non-intersecting way. Think of $S^1$ as the interval $[0,2n]$ with identified endpoints. We can assume that the $2n$ arcs of the dividing curves on $A$ intersect both $S^1\times \{0\}$ and $S^1\times \{1\}$ in $2n$ equally spaced points, thus the two boundary components of $A$ are  partitioned into $2n$ arcs.
For an arc between the two boundary components of $A$, define its \emph{slope} as the slope of the line on the cylinder
$$\faktor{[0,2n]\times [0,1]}{(\{0\}\times [0,1])\sim (\{2n\}\times[0,1])}$$
representing this arc. We work in a tight contact structure, so by Giroux's Criterion (Theorem \ref{thm:GirouxCriterion}) the dividing curve on $S$ is connected. From this, we also know that the slope $s$ of the dividing curves on $A$ must be relatively prime to $n$.

Then, for $k> 2$, we can choose $\gamma_1,\gamma_2,...,\gamma_{2k}$ so that their slopes are
$$2, 0,...,0,n, n-2,$$
respectively, see Figure \ref{fig:slopes_pspt}.


\begin{figure}[h!]
	\centering
	\includegraphics[width=1\textwidth]{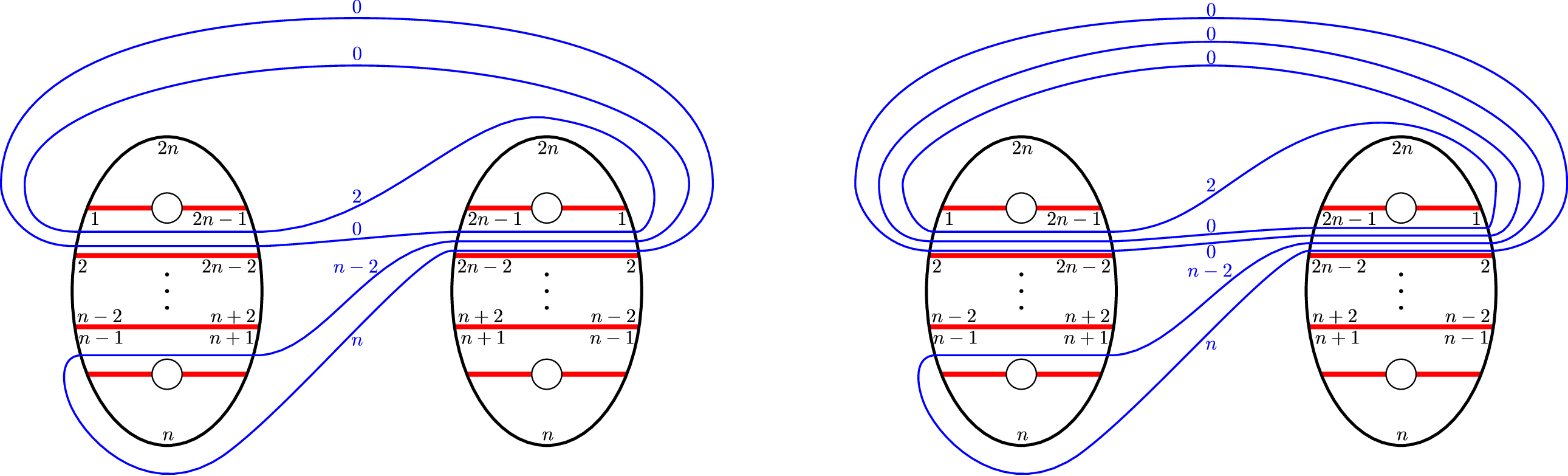}
	\caption{The blue arcs between the regions $P_1$ and $P_2$ are $\gamma_1,\gamma_2,...,\gamma_{2k}$, and the blue number on them is their slope. The $2n$ arcs on the boundary of $P_i$ determined by the dividing curves are numbered with black. For $k> 2$, we can isotope $\gamma_1,\gamma_2,...,\gamma_{2k}$ so that their slopes are $2, 0,...,0, n, n-2$. This figure illustrates the cases when $k=3$ (left) and $k=4$ (right). }
	\label{fig:slopes_pspt}
\end{figure}

Similarly, we define a curve $\gamma'$ in $S$ for $\L$ such that it bounds a disk $\Din$ in $\Bin$ disjoint from $\L$. Depending on the number $m$ in $K(k,m)$, let $\gamma'_0$ be the curve shown in Figure \ref{fig:gammaprime}, and define $\gamma'$ as the image of $\gamma'_0$ after some number of Dehn twists on $S$ along the curve $t'$. The signed number of Dehn twists is $\frac{m}{2}$ when $m$ is even and $\frac{m-1}{2}$ when $m$ is odd, with negative numbers denoting negative Dehn twists.

\begin{figure}[h!]
	\centering
	\includegraphics[width=0.3\textwidth]{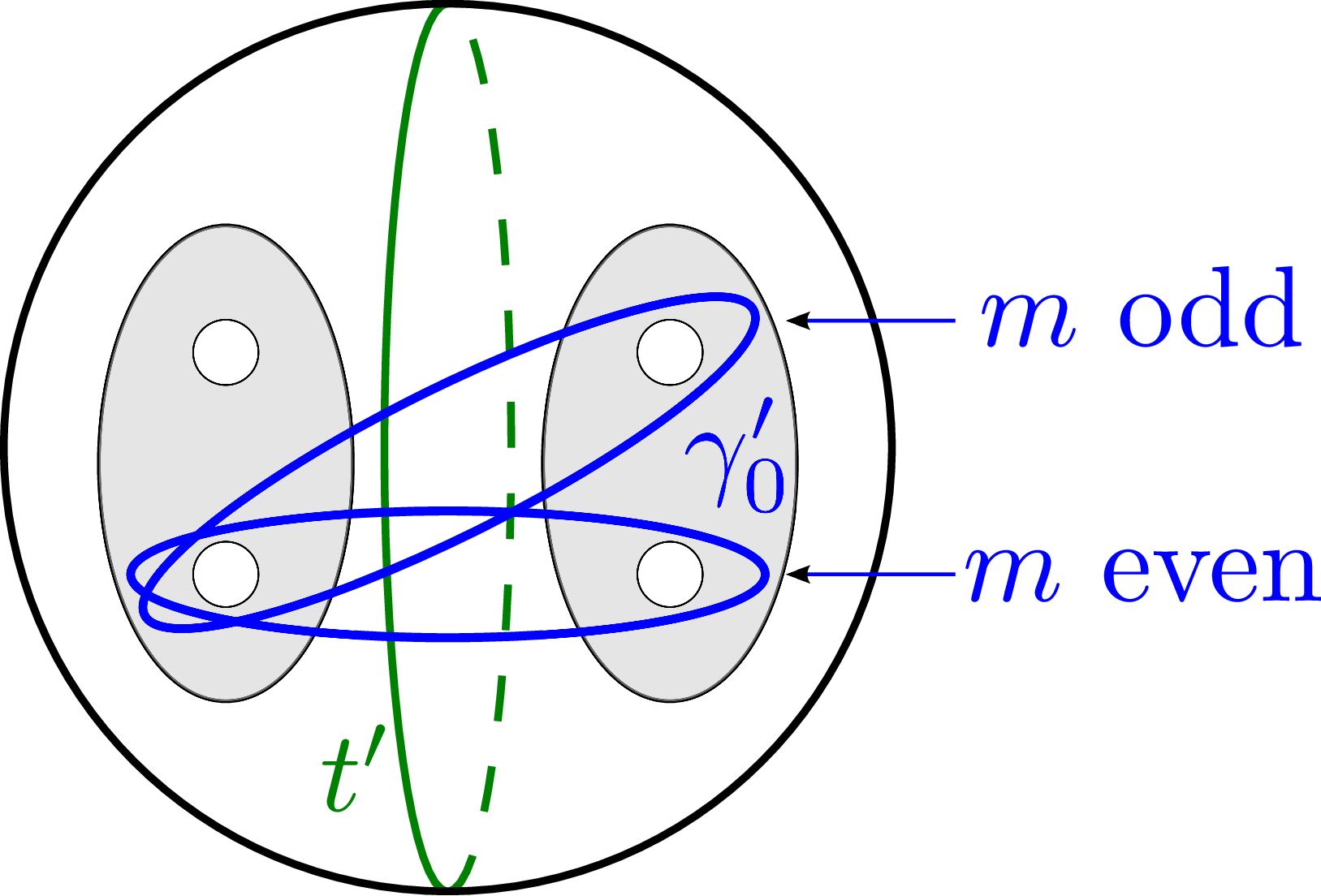}
	\caption{The curves $\gamma'_0$ and $t'$ in $S$. The shaded regions are $P_1$ and $P_2$. The definition of $\gamma'_0$ depends on the number $m$ in $K(k,m)$.}
	\label{fig:gammaprime}
\end{figure}

Again, we can assume that all the intersection points of $\gamma'$ and the dividing curves of $S$ are in $A$. Observe that $\gamma'$ always intersects the annulus $A$ in two parallel arcs $\gamma'_1$ and $\gamma'_2$ of slope $nm$. (If $m$ is even, then $2n\cdot \frac{m}{2}$, if $m$ is odd, then $n+2n\cdot \frac{m-1}{2}$.)

Legendrian realize $\gamma$ and $\gamma'$ and make $\Dout$ and $\Din$ convex without moving $\gamma$ and $\gamma'$. If $\tb(\gamma)$ or $\tb(\gamma')$ is not maximal, then we can destabilize the curve by finding boundary parallel dividing arcs on $\Dout$ or on $\Din$. We can also notice, that in this case, $\Dout$ or $\Din$ has at least two boundary parallel dividing arcs, since a disk cannot have only one boundary parallel dividing arc. This means that we can find at least two bypasses for $S\setminus N_{\U}$, disjoint from $\L$.
Let $c$ denote a curve along which such a bypass is attached. 
Recall that by Lemma \ref{lem:GammaPi}, the dividing curves on $P_i$ are arranged as in Figure \ref{fig:Fig19ENV}, and that $\gamma$ and $\gamma'$ intersect the dividing curves of $S$ only on the annulus $A$. Therefore, if $c\cap (P'_1\cup P'_2)$ has at most one component, then we can apply the following result. Claim 4.3 and Remark 4.4 of Etnyre, Ng and
Vértesi \cite{ENV}, together with the intersection argument
below, give the statement below in our setting:

\begin{proposition}\cite{ENV}\label{Claim43}
If $n>1$ and there is a bypass attached from inside or outside along a curve $c$ that intersects $P'_1\cup P'_2$ in at most one component then $\U$ can be destabilized.
\end{proposition}
\begin{remark}
	In the proof of Proposition \ref{Claim43}, the following statements \mbox{are verified}: 
	\begin{itemize}
		\item if $n\neq 3$ then $\U$ can be destabilized,
		\item if $n=3$ then either $\U$ can be destabilized or the slope $s$ of the dividing curves on $A$ can be increased by 1 if $c$ is attached from outside and decreased by 1 if $c$ is attached from inside. After changing $s$ once, any further bypasses attached from the same side of $S$ will yield a destabilization of $\U$. Moreover, we can always find two consecutive bypasses along arcs of $\gamma$ or $\gamma'$. For $n=3$, we use $\gamma'$ when $|s-3m|\geq3$, and $\gamma$ otherwise. In the latter case, $s\geq7$ or	$s\leq-4$. The required intersection bounds remain valid after the first exceptional bypass in either case. Thus, also in this case, $\U$ can be destabilized.
	\end{itemize}
For the details see \cite[Claim 4.3.]{ENV}.
\end{remark}


\begin{corollary}\label{Claim43Corollary}
	If $c\cap (P'_1\cup P'_2)$ has at most one component and $n>1$, then (possibly after changing $s$ once) we can find bypasses such that $\U$ destabilizes, contradicting the maximality of $\tb(\U)$.
\end{corollary}

Let us examine how many components $c\cap (P'_1\cup P'_2)$ has.

\begin{lemma}\label{lem:legfeljebb1komponens}
Assume $n>1$. There is a bypass (attached from inside or outside) along a curve $c$ that intersects $P'_1\cup P'_2$ in at most one component.
\end{lemma}
\begin{proof}
We denote the dividing curves on $A$ by $\Gamma_A$. Recall that $\gamma$ bounds a disk in $\Bout$, while $\gamma'$ in $\Bin$. Therefore, a bypass along a curve $c$ on $\gamma$ is a bypass from outside, while on $\gamma'$ is a bypass from inside.

To count the components of $c\cap (P'_1\cup P'_2)$, we will examine the number of intersection points of $\Gamma_A$ and $\gamma$ or $\gamma'$. In particular, we will show that for any possible $m$, $n$ and $s$, either
\begin{itemize}
	\item all subarcs $c$ of $\gamma$ which intersect $\Gamma_A$ in three points (that is, can be an attaching arc of a bypass) intersect $P'_1\cup P'_2$ in at most one component, or
	\item all subarcs $c$ of $\gamma'$ which intersect $\Gamma_A$ in three points (that is, can be an attaching arc of a bypass) intersect $P'_1\cup P'_2$ in at most one component.
\end{itemize}

Fix an orientation on $S$, $\gamma$ and $\gamma'$. Let $R_+\subset S\setminus\Gamma_S$ denote the region where $\alpha(v)>0$, and orient $\Gamma_A$ by the boundary orientation induced from $R_+$.
Call an intersection point of $\gamma$ (or $\gamma'$) and $\Gamma_A$ \emph{positive}/\emph{negative} if the intersecting arcs of $\gamma$ (or $\gamma'$) and $\Gamma_A$ form a positive/negative basis of $S$. A consecutive positive and negative intersection point together can form a canceling pair, see Figure \ref{fig:cancelingpair}.

\begin{figure}[h!]
	\centering
	\includegraphics[width=0.15\textwidth]{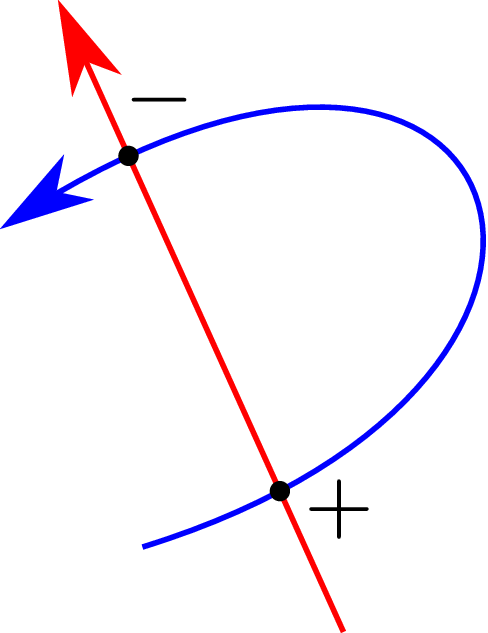}
	\caption{A canceling pair of a positive and a negative intersection point}
	\label{fig:cancelingpair}
\end{figure}
	
To work with the minimal number of intersections, we want to avoid canceling pairs in $\Gamma_A\cap \gamma$ and in $\Gamma_A\cap \gamma'$. We need to examine two cases. We will check that either

\begin{enumerate}[leftmargin=*, label=Case \arabic*:]
	\item the intersection points of $\Gamma_A$ and $\gamma$ are non-canceling and $|\Gamma_A\cap \gamma_i|\geq 2$ for $i=1,2,...2k$, or
	\item the intersection points of $\Gamma_A$ and $\gamma'$ are non-canceling and $|\Gamma_A\cap \gamma_i'|\geq 2$ for $i=1,2$.
\end{enumerate}

 These conditions suffice: an attaching arc meeting
 $P'_1\cup P'_2$ twice would have at least four
 intersections with $\Gamma_A$.
 Thus, in either case, $c\cap(P'_1\cup P'_2)$
 has at most one component.

First, let us check when Case 2 holds by studying $\Gamma_A\cap \gamma'$.
Recall, that the two parallel components $\gamma_1'$, $\gamma_2'$ of $\gamma'$ have slope $nm$, while the dividing curves of $\Gamma_A$ have slope $s$.

Since $\gamma'_1$ and $\gamma'_2$ are parallel, the intersections
of $\gamma'$ with the dividing set in $P$ are essential.
Hence, for $i=1,2$, the number of intersection points is
$|\Gamma_A\cap\gamma_i'|=|s-nm|$.

Next, we check whether for both $\gamma_1'$ and $\gamma_2'$ there are at least two intersection points on $A$, that is, if $|\Gamma_A\cap \gamma_i'|\geq 2$. As mentioned before, this would mean that $c\cap (P'_1\cup P'_2)$ has at most one component for any bypass curve $c$ along $\gamma'$.

A short analysis shows that $|\Gamma_A\cap \gamma_i'|=|s-nm|\geq 2$ with $n>1$ holds for all $|m|>2$ when $|s|\leq 3n-2$. So if $|m|>2$ our statement holds for a bypass along $\gamma'$, proving Lemma \ref{lem:legfeljebb1komponens} for the pairs $(n,s)$ shown on Figure \ref{fig:diagram_gammaprime}. If $m=-2$ then the condition $|\Gamma_A\cap \gamma_i'|=|s-nm|\geq 2$ with $n>1$ is equivalent to $|s+2n|\geq 2$ which is true for all $s\geq 0$, see Figure \ref{fig:diagram_gammaprime-2}.

\begin{figure}[h!]
	\centering
	\includegraphics[width=0.37\textwidth]{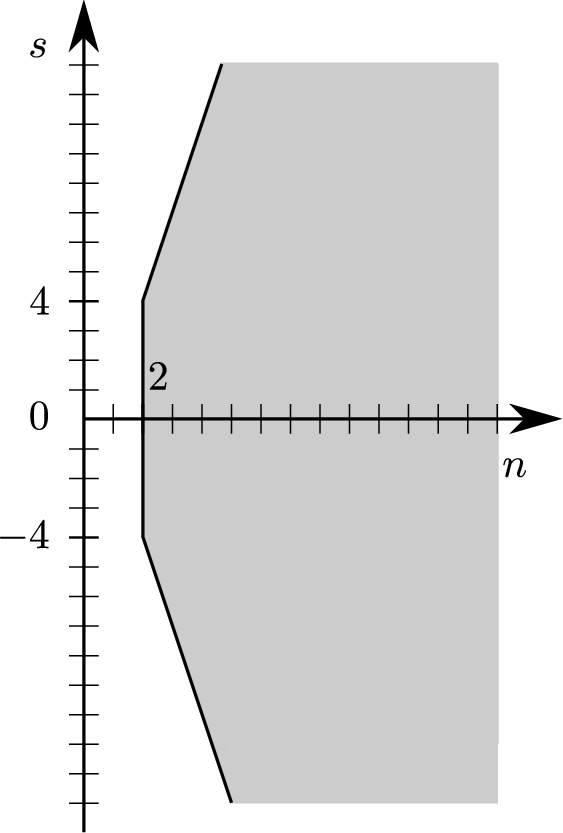}
	\caption{The $(n,s)$ pairs for which the conditions of Case 2, in particular $|\Gamma_A\cap \gamma_i'|=|s-nm|\geq 2$ with $i=1,2$ and $n>1$ hold for all $|m|>2$.}
	\label{fig:diagram_gammaprime}
\end{figure}

\begin{figure}[h!]
	\centering
	\includegraphics[width=0.37\textwidth]{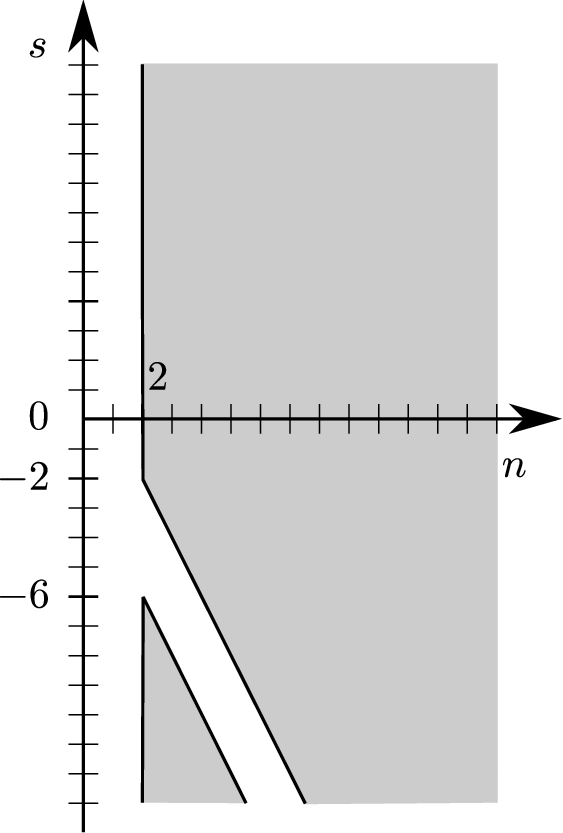}
	\caption{The $(n,s)$ pairs for which the conditions of Case 2 hold with $m=-2$. In the shaded region $|s+2n|\geq 2$ with $n>1$.}
	\label{fig:diagram_gammaprime-2}
\end{figure}

Next we study the bypasses along $\gamma$ to check when Case 1 holds.

We chose the arcs $\gamma_1, \gamma_2,...,\gamma_{2k}$ so that their slopes were $2, 0,...,0,n, n-2$. The corresponding slope differences are
$s-2$, $s,\ldots,s$, $s-n$, $s-n+2$.
If their signs agree, the intersections are essential:
a cancellation could only occur around an arc through one
of the $P'_i$, where the agreeing slope signs exclude it,
as in \cite[proof of Lemma 4.2]{ENV}. Moreover, we would like to see that whenever $|\Gamma_A\cap \gamma_i'|< 2$ for $i=1$ or $2$ (that is, the conditions of Case 2 do not hold), then each of $|\Gamma_A\cap \gamma_1|=|s-2|$, $|\Gamma_A\cap \gamma_2|=|s|$,...,$|\Gamma_A\cap \gamma_{2k-2}|=|s|$, $|\Gamma_A\cap \gamma_{2k-1}|=|s-n|$, $|\Gamma_A\cap \gamma_{2k}|=|s-n+2|$ is at least 2.

\begin{figure}[h!]
	\centering
	\includegraphics[width=0.37\textwidth]{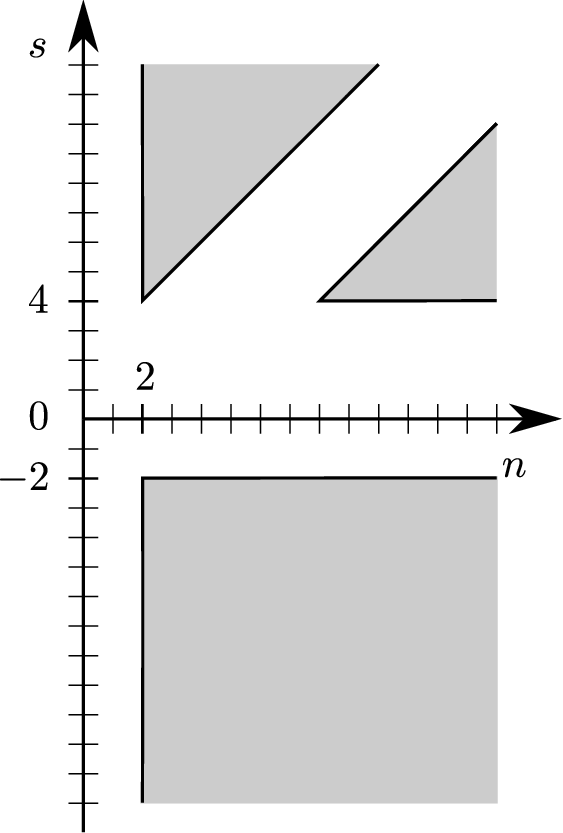}
	\caption{The pairs $(n,s)$ for which the conditions of Case 1 hold, thus $|\Gamma_A\cap \gamma_i|\geq 2$ for $i=1,2,...2k$.}
	\label{fig:diagram_gamma}
\end{figure}
The pairs $(n,s)$ for which every part of these two conditions hold are shown in Figure  \ref{fig:diagram_gamma}.

Finally, compare the $(n,s)$ pairs when we verified that Lemma \ref{lem:legfeljebb1komponens} holds for a bypass along $\gamma$ or $\gamma'$. We see that if $|m|>2$ the union of the pairs in Figure \ref{fig:diagram_gammaprime} and in Figure \ref{fig:diagram_gamma}, while if $m=-2$, the union of the pairs in Figure \ref{fig:diagram_gammaprime-2} and in Figure \ref{fig:diagram_gamma} cover all the possible $(n,s)$ pairs. This means that we proved Lemma \ref{lem:legfeljebb1komponens} in every case.
\end{proof}
\vspace{-0.5cm}

\begin{proof}[Proof of Lemma \ref{lem:Unknot}]
	Recall that we chose $\U$ to be a Legendrian knot in the topological class of $B$, which is smoothly isotopic to the unknot and disjoint from $\L$. Moreover, we assumed that $\tb(\U)=-n$ is maximal with some $n>0$. Assume that $n>1$. We only need to notice that Lemma \ref{lem:legfeljebb1komponens} ensures that the condition of Corollary \ref{Claim43Corollary} holds in every case. So we can destabilize $\U$, yielding a contradiction.
\end{proof}
We can use the above results to prove Theorem \ref{thm:normalizingGamma}.

\begin{proof}[Proof of Theorem \ref{thm:normalizingGamma}]
	According to Lemma \ref{lem:Unknot}, we can consider a Legendrian unknot $\U$ with maximal Thurston-Bennequin number, smoothly isotopic to $B$ of Figure \ref{fig:Unknot_B}. As in the proof of Lemma \ref{lem:Unknot}, using $\U$, define $S=A_1\cup D_1\cup D_2$, $P=P_1\cup P_2\cup A_1$, $A=A_1\cup A_{1,3}\cup A_{2,3}$ and $P'=P'_1\cup P'_2\cup A$, and consider the curve $\gamma$. Since $\tb(\U)=-1$, now $\Gamma_A$ intersects $\partial A$ in four points: two on $A_{1,3}$ and two on $A_{2,3}$. If we think of $S^1$ as the interval $[0,2]$ with identified endpoints, and $A=S^1\times [0,1]$, then we can assume that the intersection points of $A$ with $\Gamma_A$ are (0,0), (0,1), (1,0) and (1,1). Then the slope $s$ of the dividing curves on $A$ is an integer. Figure \ref{fig:slope01} shows the cases $s=0$ and $s=1$. Observe that any other case can be obtained from these by doing some number of Dehn twists on the sphere $S$ along a curve parallel to a component of $\partial A$.
	
	\begin{figure}[h!]
		\centering
		\includegraphics[width=0.75\textwidth]{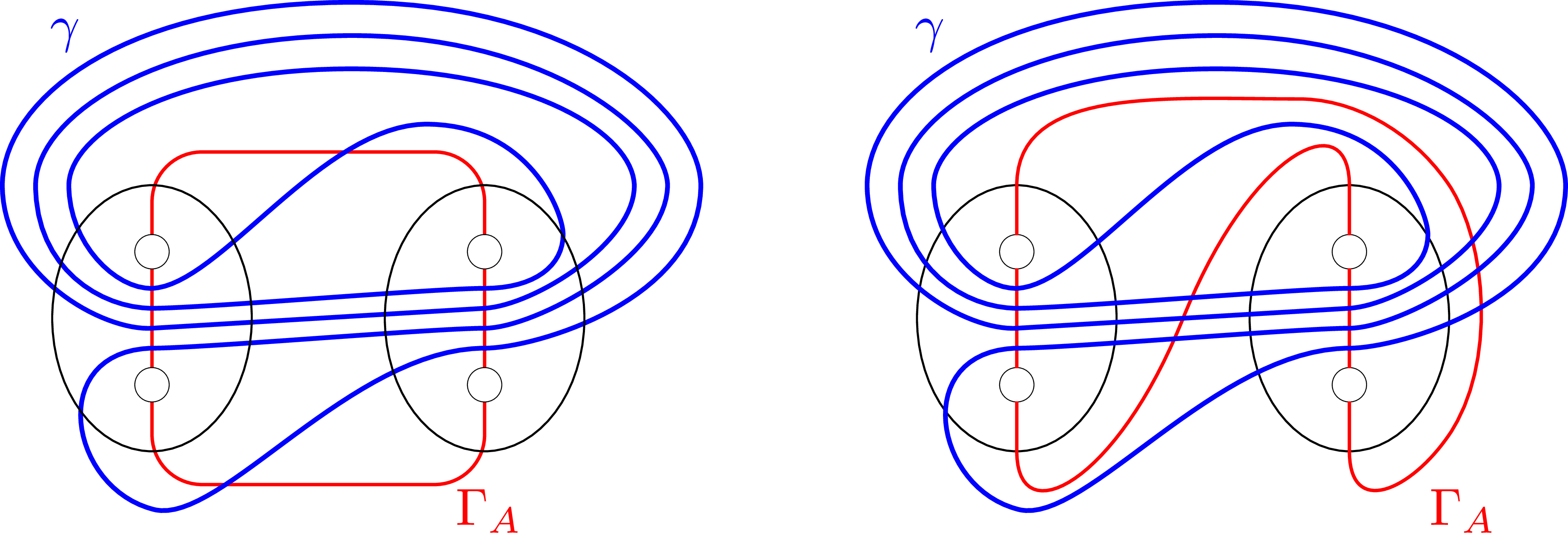}
		\caption{The cases when the slope of the red dividing curve on $A$ is $s=0$ (left) and $s=1$ (right). The thinner blue curve represents $\gamma$.}
		\label{fig:slope01}
	\end{figure}

Since $\gamma$ bounds a disk $\Dout$ in $\Bout$, we can always find a bypass for $P'$ on $\Dout$ along $\gamma$. We will prove that after some number of bypasses, $\Gamma_A$ can be made to have slope $s=1$ meaning that the dividing curves on $P'$ can be normalized as in Figure \ref{fig:dividings_on_sphere}. Our plan is to examine cases depending on the slope $s$ of $\Gamma_A$. In each case, we go through all the sub-arcs of $\gamma$ which intersect $\Gamma_A$ in three points, and check how attaching a bypass along that arc changes $s$.

\begin{figure}[h!]
	\centering
	\includegraphics[width=0.75\textwidth]{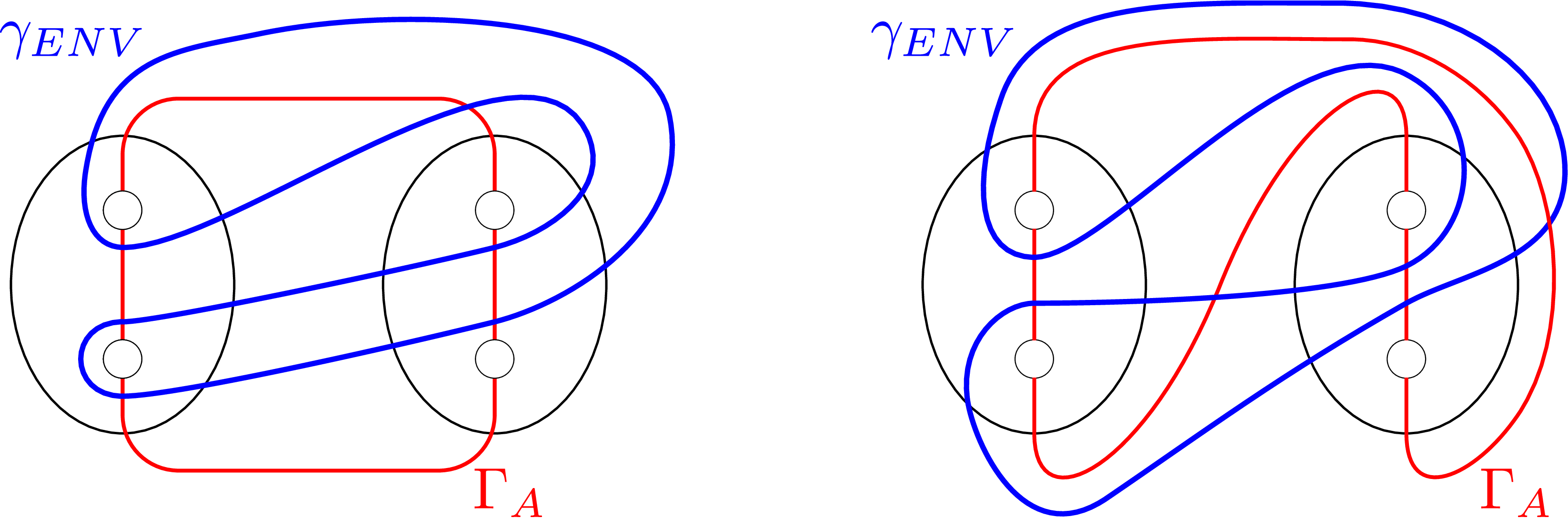}
	\caption{The curve $\gamma_{ENV}$ of \cite{ENV} and the dividing curve with slope $s=0$ (left) and $s=1$ (right). Observe how $\gamma_{ENV}$ differs from our $\gamma$ shown in Figure \ref{fig:slope01}.}
	\label{fig:slope01_Vera}
\end{figure}

In a very similar setup, Etnyre, Ng and Vértesi proved an analogous result, \mbox{\cite[Theorem 4.1.]{ENV}}. Our statement is a generalization of theirs. The only difference is that we use a different $\gamma$ curve yielding more possible bypass attachments to check (compare Figure \ref{fig:slope01} with Figure \ref{fig:slope01_Vera} showing their curve which we will denote by $\gamma_{ENV}$ with dividing slope $0$ and $1$). Therefore, we only need to check those cases which do not follow from \cite[Theorem 4.1.]{ENV}.

The positive and negative regions of $P'$ are disks. Therefore,
we compare attaching arcs up to isotopy in these regions, with
the intersection points moving along the same dividing arcs
without changing their order. Since an attaching arc contains
three consecutive intersection points of $\gamma$ with the
dividing set, the extra parallel passages for larger $k$ only
repeat the configurations considered below.

\begin{itemize}
	\item If $s\leq-1$, then Figure \ref{fig:s=-1_nalam_vs_ENV} illustrates our $\gamma$ curve together with the one used in \cite{ENV}. Although we have more intersection points of $\gamma$ and the dividing curve, we can see that the type of possible bypasses are all equivalent to the cases of \cite{ENV}. They showed that these bypasses are either disallowed, meaning that they result in disconnected dividing curves, or increase the slope by 1, allowing us to assume $s=0$.
	\begin{figure}[h!]
		\centering
		\includegraphics[width=0.75\textwidth]{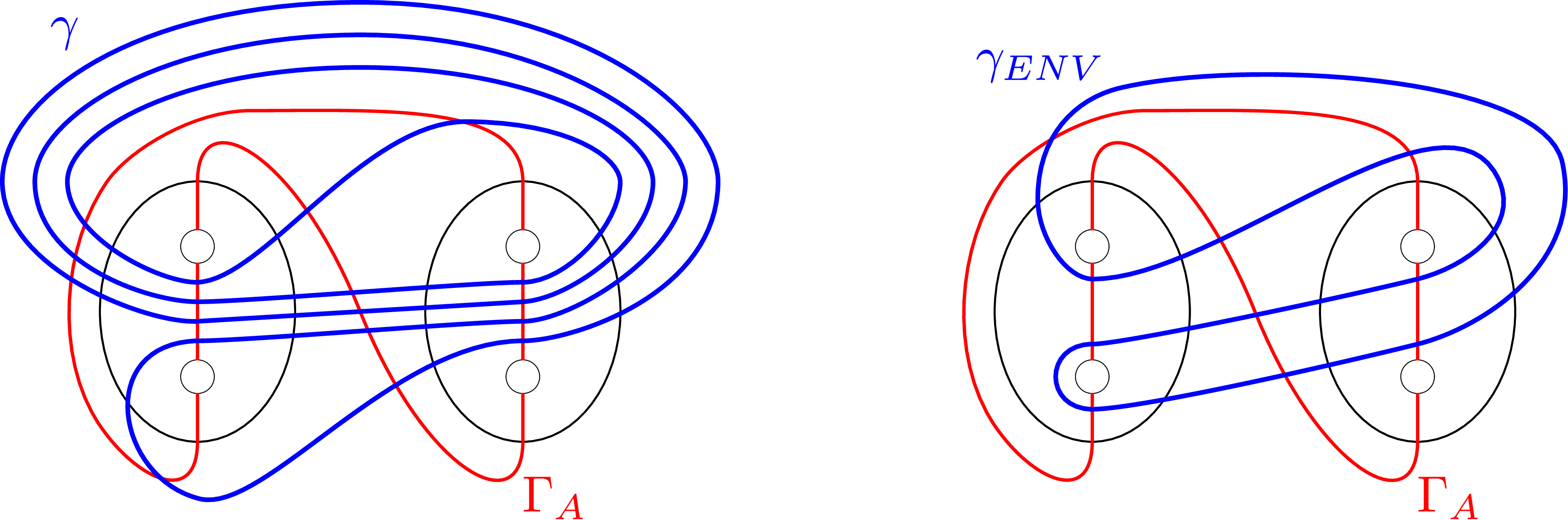}
		\caption{Our curve $\gamma$ on the left and $\gamma_{ENV}$ of \cite{ENV} on the right in case of $s= -1$. All the possible bypass attachments along $\gamma$ are equivalent to those \mbox{along $\gamma_{ENV}$}. If $s<-1$, no new type of bypass attaching curves occur.}
		\label{fig:s=-1_nalam_vs_ENV}
	\end{figure}	
	\item If $s=0$, we examine Figure \ref{fig:s=0_nalam_vs_ENV} to compare our situation with the one in \cite{ENV}. We can see that we have two new cases to consider, shown in Figure \ref{fig:s0newcases}. In both cases, after two bypass rotations and attaching
	the resulting bypass, we obtain a dividing curve with
	slope $1$, see Figure \ref{fig:rotacio}. According to \cite{ENV}, all the other cases are either disallowed or (after possibly applying bypass rotation) increase the slope by 1, allowing us to assume $s=1$.
		\begin{figure}[h!]
		\centering
		\includegraphics[width=0.75\textwidth]{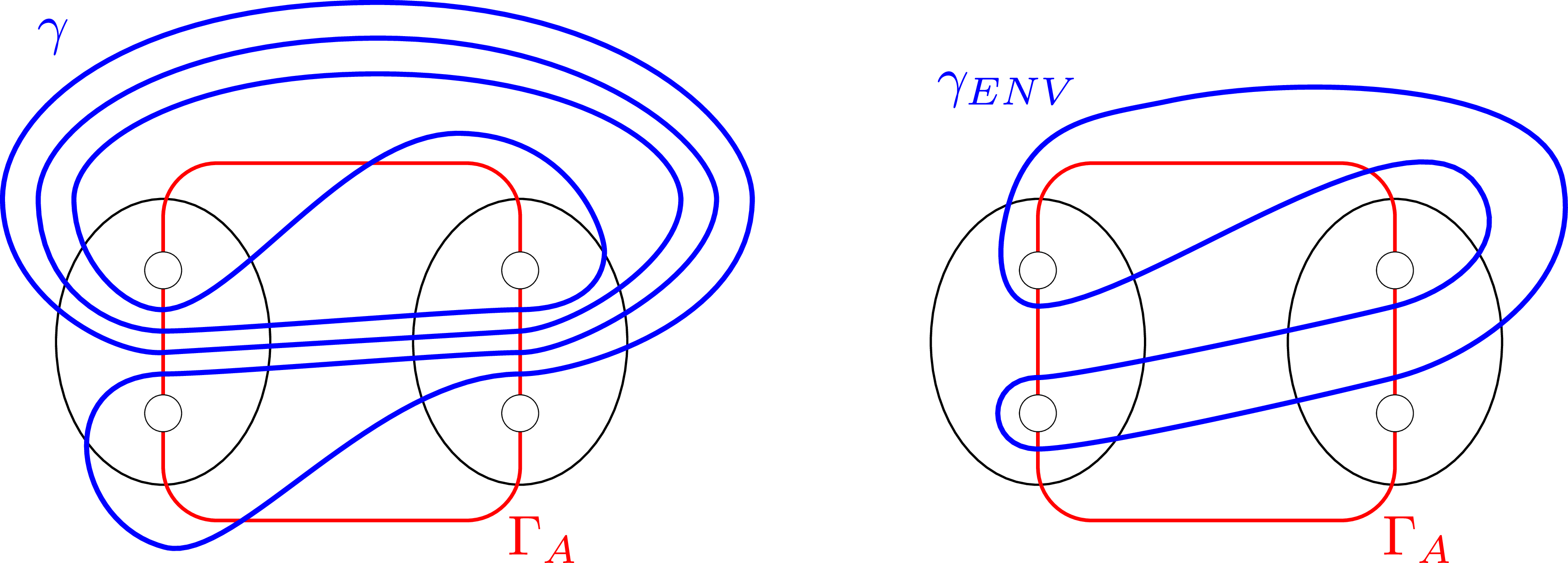}
		\caption{Our curve $\gamma$ on the left and $\gamma_{ENV}$ of \cite{ENV} on the right in case of $s=0$. There are two possible bypass attachments along $\gamma$ which did not occur \mbox{along $\gamma_{ENV}$}, see Figure \ref{fig:s0newcases}.}
		\label{fig:s=0_nalam_vs_ENV}
		\end{figure}
	\begin{figure}[h!]
		\centering
		\includegraphics[width=0.75\textwidth]{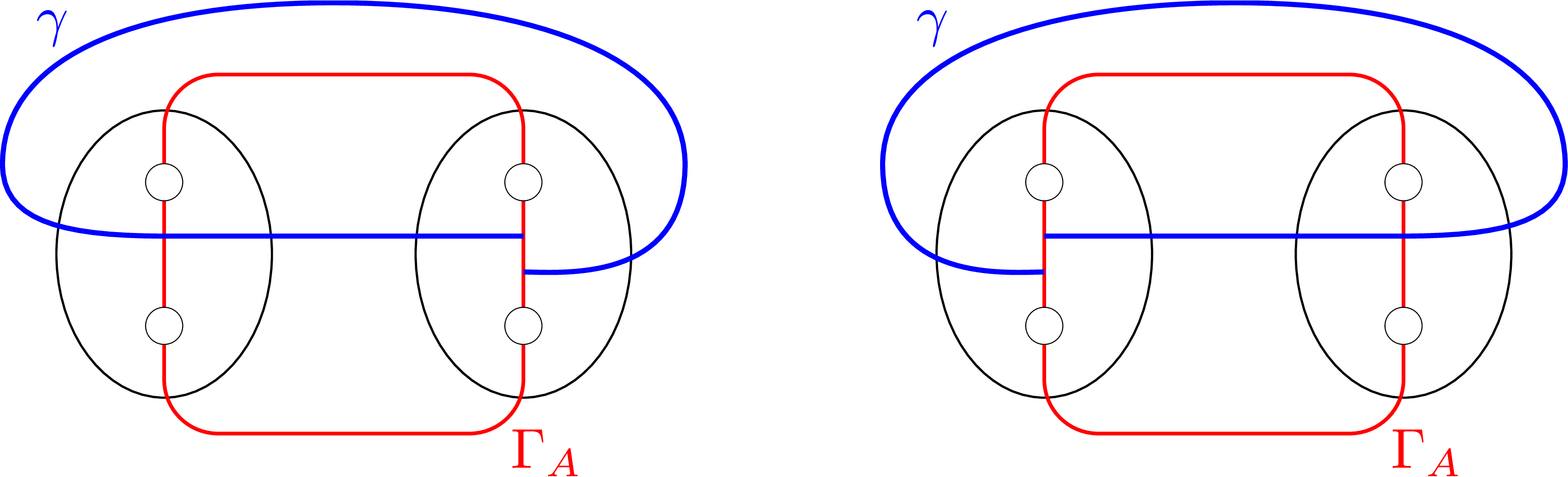}
		\caption{The two possible bypass attachments along $\gamma$ which did not occur along $\gamma_{ENV}$ in \cite{ENV} in the $s=0$ case.}
		\label{fig:s0newcases}
	\end{figure}
\begin{figure}[h!]
	\centering
	\includegraphics[width=1\textwidth]{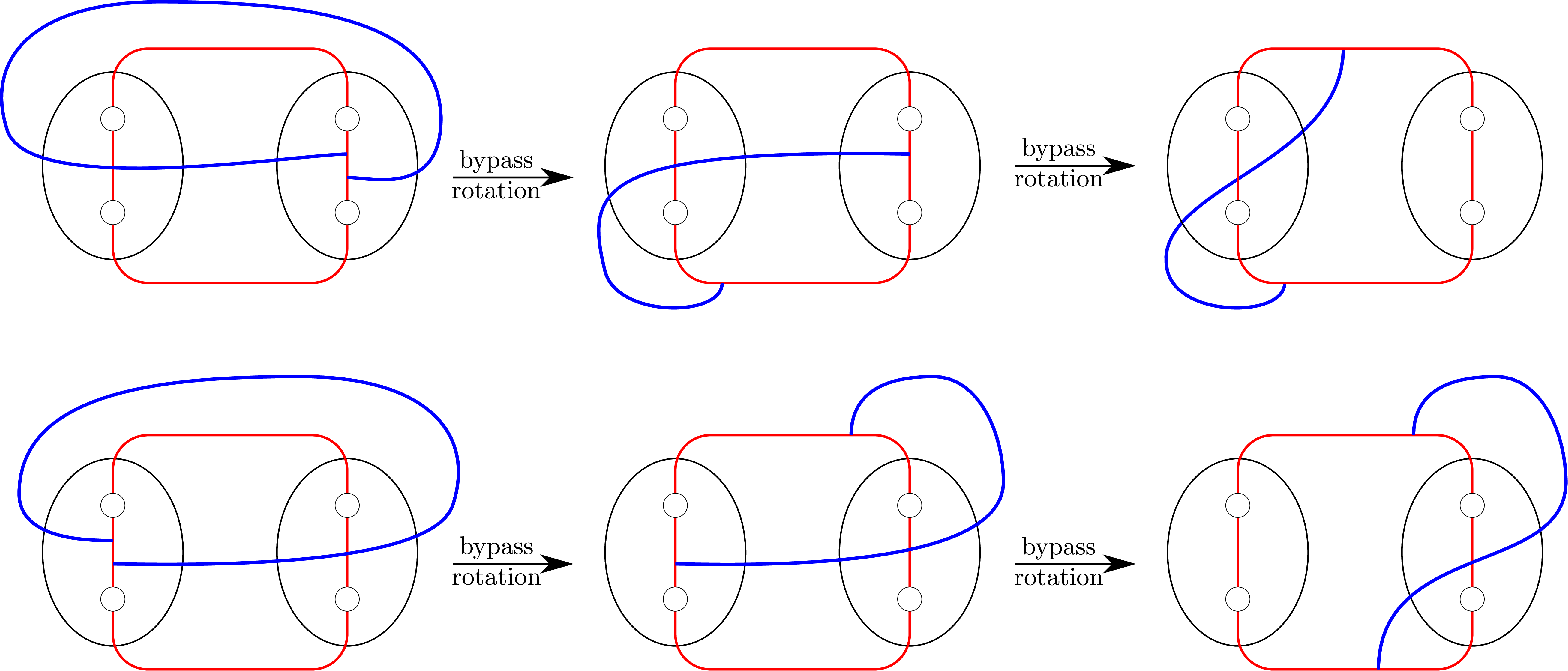}
	\caption{Handling the two cases which did not occur in \cite{ENV}. By applying two bypass rotations, the dividing curve will have slope 1.}
	\label{fig:rotacio}
\end{figure}
	\item $s=1$ means that the dividing curves are arranged as in Figure \ref{fig:dividings_on_sphere}.
\item If $s\geq2$, then Figure \ref{fig:s=2_nalam_vs_ENV}
shows that all the possible bypass attachments are of the same
types as in \cite{ENV}. The argument in
\cite[proof of Theorem 4.1]{ENV}, including their Figure 25,
again allows us to arrange $s=1$.
\end{itemize}
	\begin{figure}[h!]
	\centering
	\includegraphics[width=0.75\textwidth]{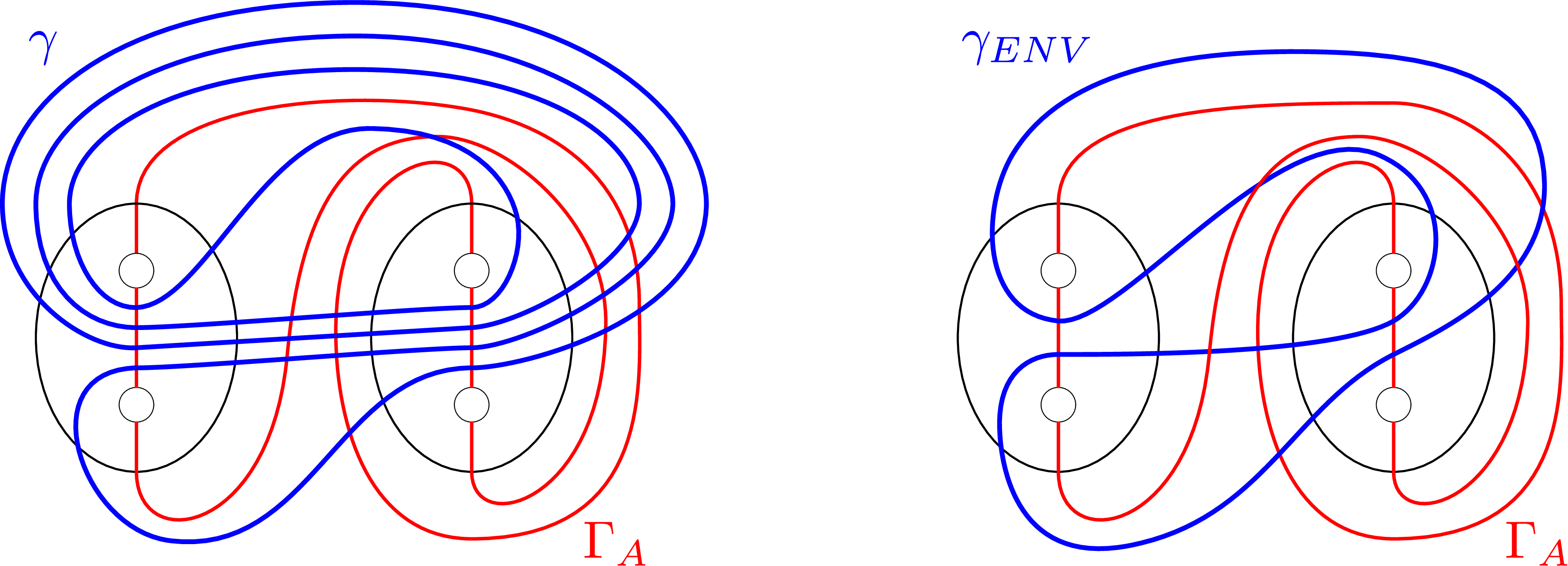}
	\caption{Our curve $\gamma$ on the left and $\gamma_{ENV}$ of \cite{ENV} on the right in case of $s= 2$. All the possible bypass attachments along $\gamma$ are equivalent to those along $\gamma_{ENV}$. \mbox{If $s>2$}, no new type of bypass attaching curves occur.}
	\label{fig:s=2_nalam_vs_ENV}
\end{figure}
Since all possible cases result in $s=1$, we see that the dividing curves on $P'$ can be normalized as in Figure \ref{fig:dividings_on_sphere}.

The last thing we need to check is that we can extend our proof from \mbox{$P'$ to $P$}. Note that they only differ in four small annular neighborhoods $A_{i,j}$ ($i,j=1,2$) of the boundary. Also, the dividing set on $P$ only differs from that of $P'$ by some number of twists on $A_{i,j}$. Recall that we defined $P_i$ as $D_i\setminus N$ for a standard contact neighborhood $N$ of $\L$. But we can consider a bigger tubular neighborhood $\widetilde{N}$ of $\L$ such that $\widetilde{N}$ contains $N$ and the four annuli $A_{i,j}$. Note that adding a small neighborhood of the annuli $A_{i,j}$ to $N$ does not affect the slope of the dividing curves on the tube. Therefore, we can choose $\widetilde{N}$ to be also a standard contact neighborhood of $\L$. Now taking $\widetilde{N}$ instead of $N$, the earlier surface $P'$ is $P$ in this new setup, and this modification does not change the slopes of the dividing curves. Therefore, it completes the proof of Theorem \ref{thm:normalizingGamma}.
\end{proof}

\subsection{Inside the sphere}\label{subsec:inside}
The classification result for our 2-braid $\K\cap \Bin$ is obtained from the following theorem of Etnyre and Vértesi \cite{Satellites}:

\begin{theorem}[Etnyre-Vértesi, \cite{Satellites}]
	Consider a convex disk $D$ with Legendrian boundary and
	a single dividing arc $\Gamma_D$, and equip $D\times I$,
	where $I=[-1,1]$, with an $I$-invariant tight contact
	structure inducing this dividing set on $D$.
	Fix $q$ points $p_1,\ldots,p_q$ on $\Gamma_D$.
	Then any Legendrian $q$-braid in $D\times I$, straight
	near the boundary and with endpoints
	$\{(p_i,-1),(p_i,1)\}_{i=1}^q$, is Legendrian isotopic
	to a concatenation of the basic building blocks of
	Figure \ref{fig:braid_building_blocks}.
\end{theorem}

\begin{figure}[h!]
	\centering
	\includegraphics[width=0.9\textwidth]{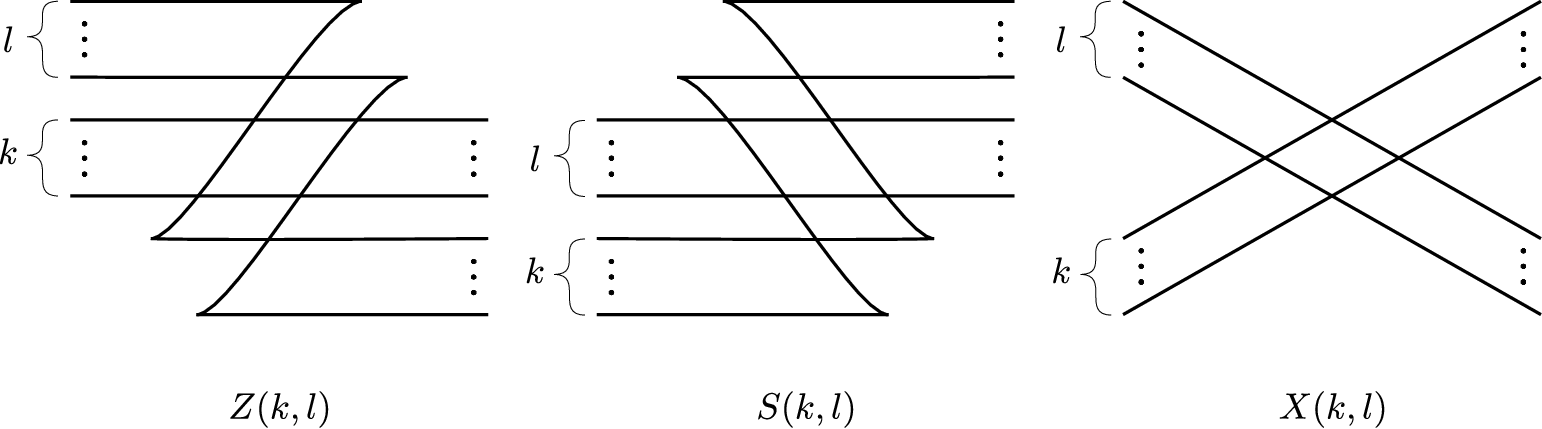}
	\caption{The three basic building blocks of Legendrian braids. Above or below the strands depicted here, there may be further horizontal ones.}
	\label{fig:braid_building_blocks}
\end{figure}
By Giroux's Flexibility Theorem and Theorem \ref{thm:B3unique},
we identify $\Bin$ with this standard contact $D\times I$,
matching the normalized boundary data and the endpoint germs.
The Legendrian isotopies below are relative to these germs.
\begin{corollary}
	A Legendrian 2-braid has three basic building blocks which we will denote by $Z=Z(1,1)$, $S=S(1,1)$ and $X=X(1,1)$.
\end{corollary}

Observe in Figure \ref{fig:XZ_stab} that a block $X$ followed by a $Z$ or an $S$ (or vice versa) is a stabilization. Therefore, the following proposition is immediate.

\begin{figure}[h!]
	\centering
	\includegraphics[width=0.5\textwidth]{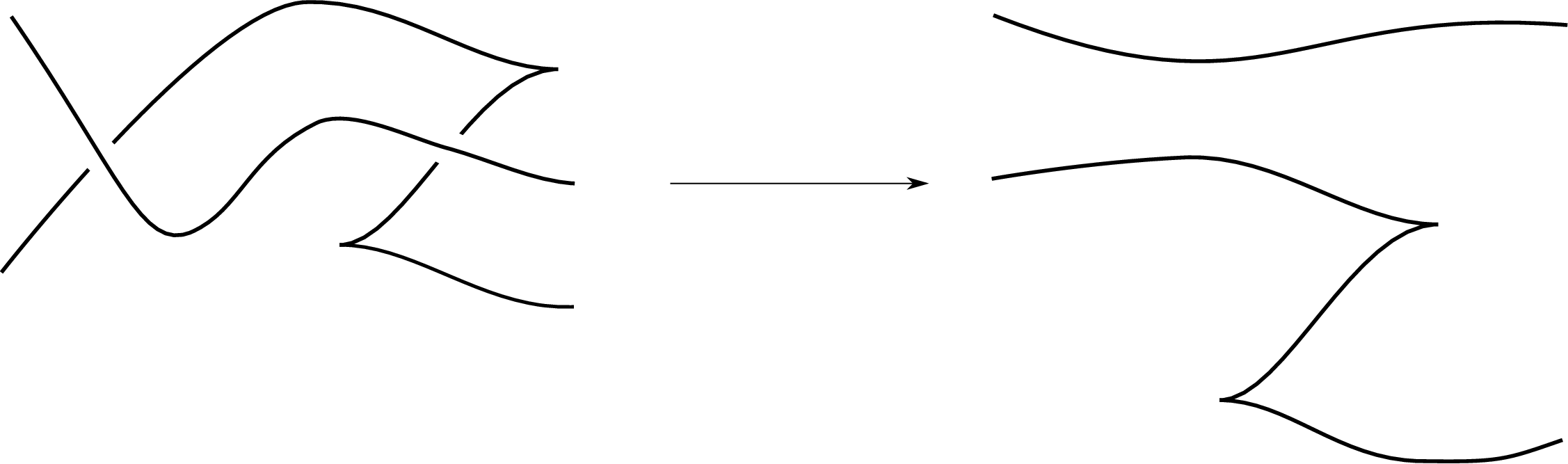}
	\caption{A block $X$ followed by a $Z$ destabilizes.}
	\label{fig:XZ_stab}
\end{figure}

\begin{proposition}\label{prop:2-braid_XSZ}
	For a 2-braid with $|m|$ half-twists,
	\begin{itemize}
		\item if $m\geq 0$ then a Legendrian realization of the braid either destabilizes or consists of $m$ blocks of type $X$;
		\item if $m<0$ then a Legendrian realization of the braid either destabilizes or consists of $|m|$ blocks of type $S$ and $Z$ in any order.
	\end{itemize}
\end{proposition}

Since $\K\cap \Bin$ is a Legendrian 2-braid with $m$ half-twists, now we can give an upper bound on the number of its non-destabilizable Legendrian realizations.


\begin{theorem}\label{thm:belul}
	The number of non-destabilizable Legendrian realizations of \mbox{$\K\cap \Bin$ is}
	\begin{itemize}
		\item exactly 1, if $m\geq 0$, and
		\item at most $\lfloor\frac{(|m|+2)^2}{4}\rfloor$ if $m<0$ .
	\end{itemize}
\end{theorem}
\begin{proof}
	The $m\geq0$ case follows from Proposition \ref{prop:2-braid_XSZ}.
	
	For $m<0$, Proposition \ref{prop:2-braid_XSZ} implies $2^{|m|}$ possible realizations. Moreover, \mbox{Figure \ref{fig:SSZ=ZSS}} shows that if we have three consecutive basic building blocks of type $SSZ$, it is Legendrian isotopic to the consecutive building blocks $ZSS$. 
	
	\begin{figure}[h!]
		\centering		\includegraphics[width=0.95\textwidth]{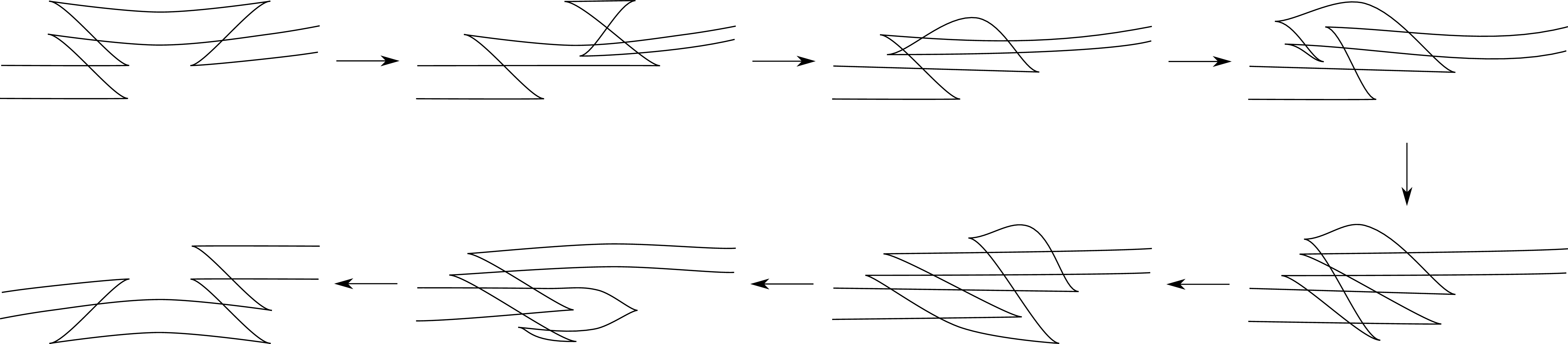}
		\caption{$SSZ$ is Legendrian isotopic to $ZSS$}
		\label{fig:SSZ=ZSS}
	\end{figure}

	Rotating the figure by 180° about the $z$ axis, we can observe that $SZZ$ is \mbox{Legendrian} isotopic to $ZZS$. Putting these together, if we number the crossings of $\K\cap \Bin$ from 1 to $|m|$, then we can swap any two building blocks with indices of the same parity. Thus the number of $S$'s on the even and on the odd indices determine the possible number of realizations. The number of $S$'s on the even indices is in $\{0,...,\lfloor \frac{|m|}{2}\rfloor\}$, while on the odd indices in $\{0,...,\lceil\frac{|m|}{2}\rceil\}$, which means we have $\left(\lfloor \frac{|m|}{2}\rfloor+1\right)\cdot \left(\lceil\frac{|m|}{2}\rceil+1\right)=\lfloor\frac{(|m|+2)^2}{4}\rfloor$ cases.
\end{proof}

\subsection{Outside the sphere }\label{subsec:outside}
 
The aim of the present section is to bound the possible tight
	contact structures on the complement of a neighborhood of
	$\K\cap\Bout$ in $\Bout$ which can arise from maximal $\tb$
	Legendrian realizations of $K(4,m)$.
	
	The dividing set on the four-punctured sphere is fixed by
	Theorem \ref{thm:normalizingGamma}. The maximal $\tb$ condition
	also restricts the dividing set on the boundary of a standard
	neighborhood of the knot. We use these boundary conditions
	to bound the possible tight contact structures on the outer
	complement. In the proof of the main theorem, these outer
	pieces will be glued to the corresponding inner pieces,
	and Proposition \ref{prop:elegctctotszamolni} will then be
	applied to the resulting closed Legendrian knots.

The following proposition is proved in \cite{legsur} as Theorem 2.13.

\begin{proposition}\cite{legsur}\label{prop:elegctctotszamolni}
In $(S^3, \xi_{std})$ two unoriented Legendrian knots are Legendrian isotopic if and only if the complements of
their standard contact neighborhoods are contactomorphic.
\end{proposition}
	
For oriented knots, the same conclusion holds if the
contactomorphism preserves the oriented Seifert longitude.

Let $\L$ be a maximal $\tb$ Legendrian realization of $K(k,m)$
and let $N$ be a standard contact neighborhood of $\L$.
The value $\tb(\L)$ restricts the dividing set on
$\partial(S^3\setminus(\Bin\cup N))$.
We will use these restrictions to bound the possible tight
contact structures on $S^3\setminus(\Bin\cup N)$ arising
from maximal $\tb$ realizations.

We use the fixed smooth model to keep the boundary markings
and the gluing along $P=S\setminus N$ fixed.
By Giroux's Flexibility Theorem, for each dividing set under
consideration we fix compatible characteristic foliations
on the relevant convex surfaces, with a common contact germ
near $P$. Contact structures are considered up to isotopy
fixed near the boundary.

		\begin{figure}[!h]
		\centering
		\includegraphics[width=0.45\textwidth]{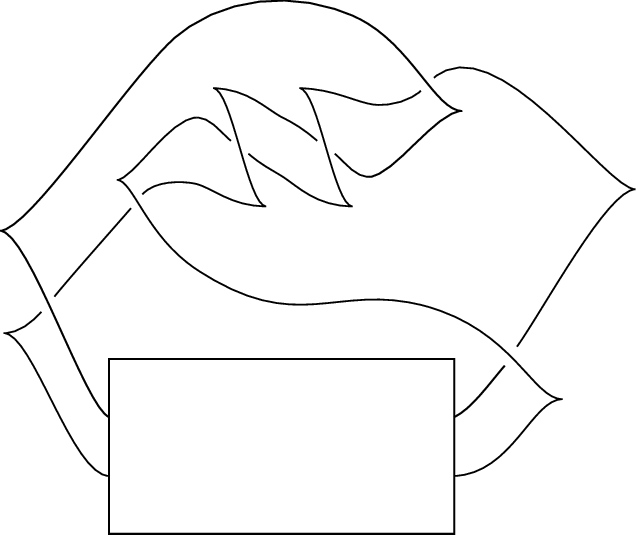}
		\caption{The outer part of a maximal $\tb$ Legendrian
			realization of $K(4,m)$.}
		\label{fig:one_ruling}
	\end{figure}

	Before stating the outer bound, we verify a lemma which will be used in the proof of the upper bound.
	
	\begin{lemma}\label{lem:outside_max_tb}
		The Legendrian realization in Figure \ref{fig:one_ruling} can be obtained as the intersection of $\Bout$ with a maximal tb Legendrian realization of $K(4,m)$.
	\end{lemma}
	\begin{proof}
			The proof is an application of Legendrian rulings. First, consider the Legendrian unknot $\U$ linked with the realization of $K(4,m)$ as in Figure \ref{fig:unknot_henger}. By Theorem \ref{thm:normalizingGamma}, we can isotope the corresponding sphere and a standard neighborhood of the Legendrian knot so that the dividing set on the four-punctured sphere is in the normalized position of	Figure \ref{fig:dividings_on_sphere}. We use this normalized position
			to identify the two parts of the knot with the inner and outer 2-braids shown in our standard model.
			
			In Figure \ref{fig:max_tbk_ruling}, we gave front diagrams of
			Legendrian knots realizing $K(4,m)$ for every allowed $m$. 
			In each of the cases considered there, the intersection
			of the given front with $\Bout$ is the Legendrian realization
			shown in Figure \ref{fig:one_ruling}.
			The normal rulings indicated in these diagrams show,
			by Proposition \ref{prop:ruling_maxtb}, that the corresponding
			Legendrian realizations have maximal $\tb$.
						
			Therefore, the Legendrian realization in Figure
			\ref{fig:one_ruling} can be obtained as the intersection of $\Bout$
			with a maximal $\tb$ Legendrian realization of $K(4,m)$.
	\end{proof}

\begin{figure}[h!]
	\centering
	\includegraphics[width=0.45\textwidth]{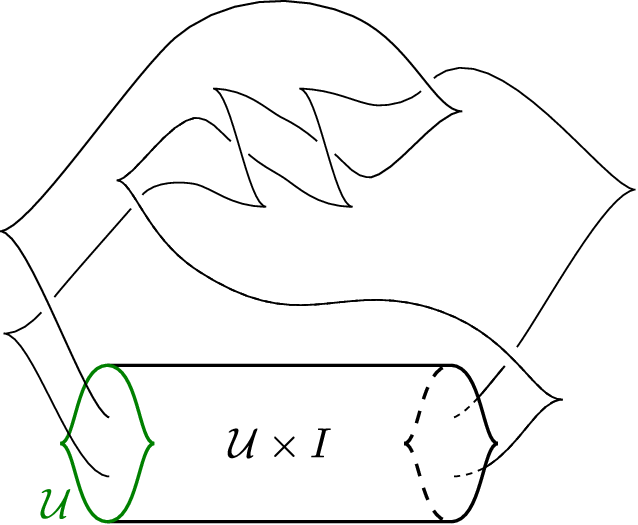}
	\caption{The cylinder $\U\times I$ which we obtain from a Legendrian unknot $\U$ linked with the Legendrian realization of $K(4,m)$.}
	\label{fig:unknot_henger}
\end{figure}

	Now we state our upper bound for the possible tight contact
	structures on the outer complement for $k=4$.

\begin{theorem}\label{thm:max10_kivul}
		Fix $k=4$ and assume that either $|m|>2$ or $m=-2$.	
		Then maximal $\tb$ Legendrian realizations $\L$ of $K(4,m)$
		determine at most $10$ tight contact structures on
		$\Bout\setminus N(\L\cap\Bout)$, up to isotopy fixed near
		the boundary.
		
\end{theorem}

\begin{proof}
Consider the standard model of $K(4, m)$, the normalized sphere $S$ and the neighborhood $\nu(K(4,m))$ as in Theorem \ref{thm:normalizingGamma} and in Figure \ref{fig:K(k,m)_and_sphere}. Let $\L$ be a maximal $\tb$ Legendrian realization of $K(4,m)$ and $N$ a standard neighborhood \mbox{of $\L$}. Define $l_1$ and $l_2$ as the two Legendrian arcs of $\L \cap \Bout$ and their neighborhoods $N_1$, $N_2$ as the components of the neighborhood $N\cap \Bout$. 
Observe that $\Bout \setminus (N_1 \cup N_2)$ is a handlebody of genus two.
Choose a complete system of two disjoint compressing disks $D_1$ and $D_2$
which cuts $\Bout \setminus (N_1 \cup N_2)$ into a 3-ball.
We choose $D_1$ as in Figure \ref{fig:diskhatar}.
The disk $D_2$ is obtained by carrying out the analogous construction
simultaneously along the second strand $l_2$.

\begin{figure}[h!]
	\centering
	\includegraphics[width=0.9\textwidth]{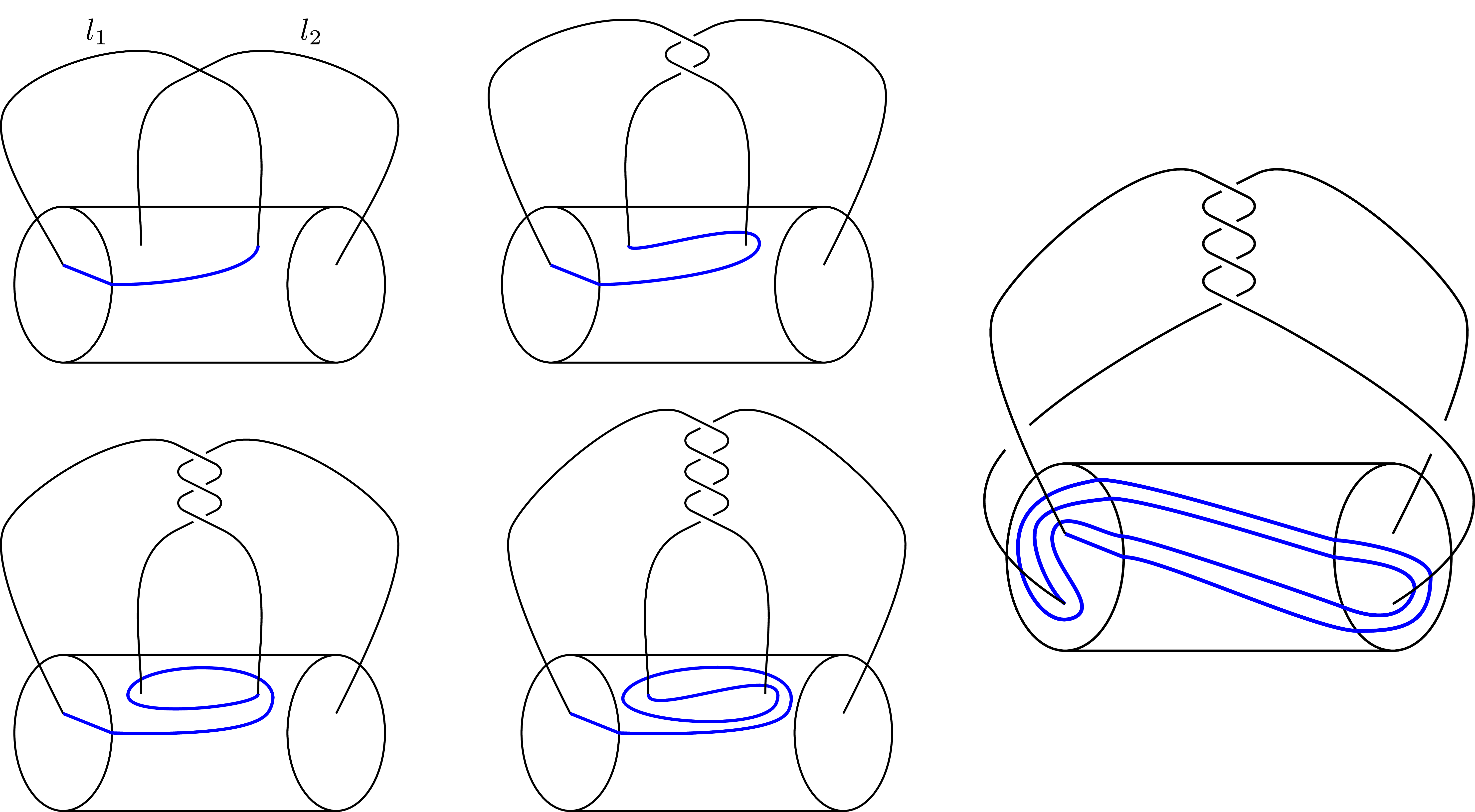}
	\caption{On the top left, the blue arc and the left black arc $l_1$
		together bound a disk. The other figures show how isotoping these arcs in $\Bout$ gives the boundary of the compressing disk $D_1$.
		The compressing disk $D_2$ is obtained by the analogous construction along the second strand $l_2$.}
	\label{fig:diskhatar}
	
\end{figure}

Choose the longitudinal parts $\partial D_i\cap\partial N_i$
to intersect the dividing set efficiently, that is, with the minimal
geometric intersection number in their relative isotopy classes with
fixed endpoints. Use the Legendrian Realization Principle,
Theorem \ref{thm:LegendrianRealizationPrinciple}, to Legendrian realize
$\partial D_i$ by isotoping $S$ (we keep the same notations).
Then perturb $D_i$, keeping its boundary fixed, to make it convex. In Figure \ref{fig:hengermetszespontok4}, we can see that $\partial D_1$ intersects the dividing curve on $\partial\Bout$ in 5 points. The analogous construction for $D_2$ gives the same number of intersection points. Thus the number of intersection points of $\partial D_i$ and the dividing curve on $\partial N_i$ also has to be odd. Let $\bar{x}_i$ denote this number. By the smoothing of the dividing set near $\partial \Bout \cap \partial N_i$ there are 2 such intersection points on $\partial N_i$ near $\partial\Bout$. As a conclusion, we get that $\partial D_i$ intersects the dividing curve on $\partial N_i$ in at least 3 points.

\begin{figure}[h!]
	\centering
	\includegraphics[width=0.9\textwidth]{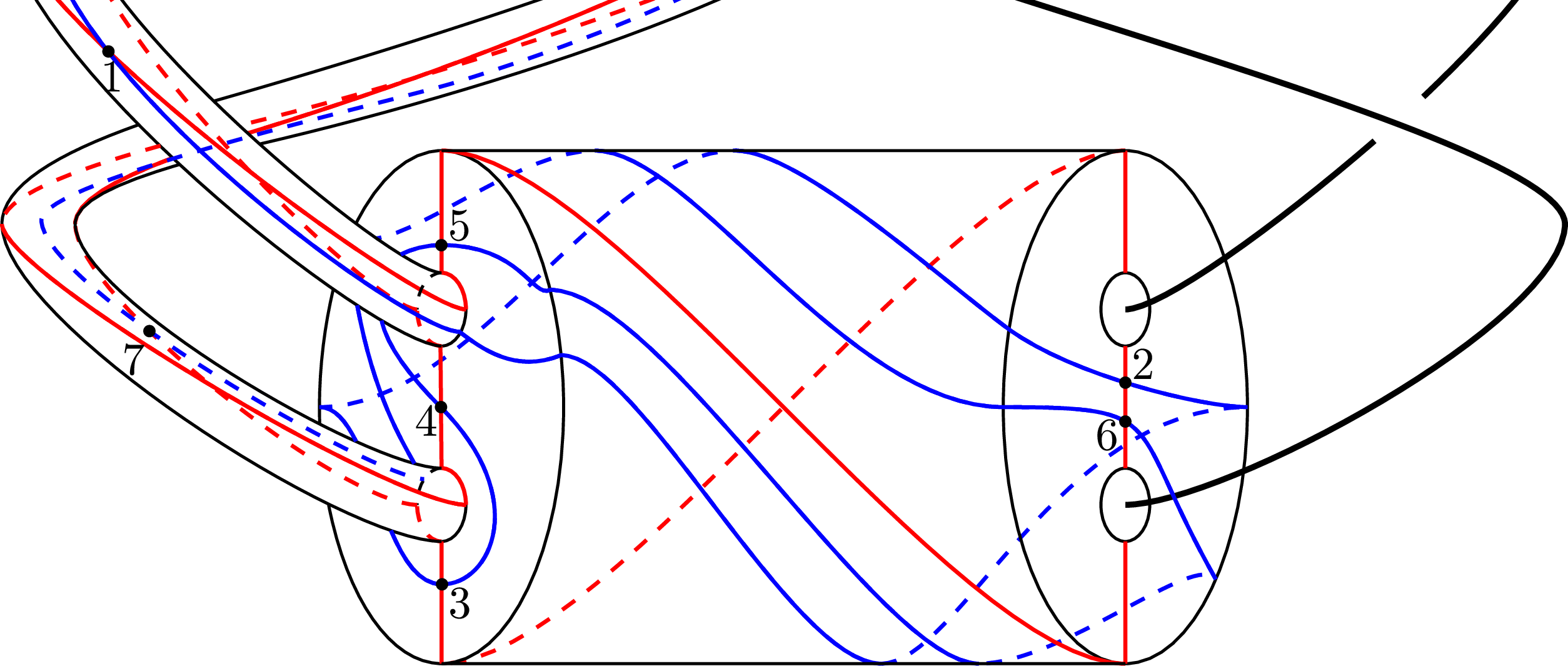}
	\caption{The boundary of the compressing disk $D_1$ and the dividing curve on $\partial \Bout$ intersect in 5 points. The analogous construction for $D_2$ gives the same number of intersection points.}
	\label{fig:hengermetszespontok4}
\end{figure}

 We now examine the tight contact structure on
 $\Bout\setminus(N_1\cup N_2)$.
 Since $D_1$ and $D_2$ cut this handlebody into a $3$-ball
 and the dividing set on $\partial\Bout$ is fixed, it remains
 to study the possible dividing sets on $D_1$ and $D_2$.
 By Theorem \ref{thm:B3unique}, a prescribed connected dividing
 set on the boundary of the cut-open ball, with the contact
 germ fixed as above, determines a unique tight contact
 structure up to isotopy fixed near the boundary.
 Gluing back along the fixed disks therefore gives at most
 one outer contact structure for each pair of dividing set
 configurations.

 We have already observed that $\bar{x}_i$ ($i=1,2$) is at least 3
 and odd. Let us denote these intersection points of $\partial D_i$
 and the dividing curve on $\partial N_i$ by $s_0,s_1,...,s_{2j}$.
 There is no bypass between the middle intersection points
 $s_1,...,s_{2j-1}$ because it would be a bypass along $\partial N_i$
 giving a destabilization of $\L$. For the number of these middle
 intersection points, we use the notation
 $x_i=\bar{x}_i-2=2j-1$.
 Since $\partial D_i$ intersects the dividing curve on
 $\partial\Bout$ in 5 points, the total number of intersection
 points of $\partial D_i$ and the dividing set is $x_i+7=2j+6$.
 After rounding the edges, the absence of such a bypass forces
 the $x_i+1$ consecutive endpoints of the dividing set on $D_i$
 along the middle part to be joined to the remaining six endpoints.
 Thus $x_i\leq5$, and hence $8\leq2j+6\leq12$.

 \begin{figure}[h]
 	\centering
 	\includegraphics[width=1\textwidth]{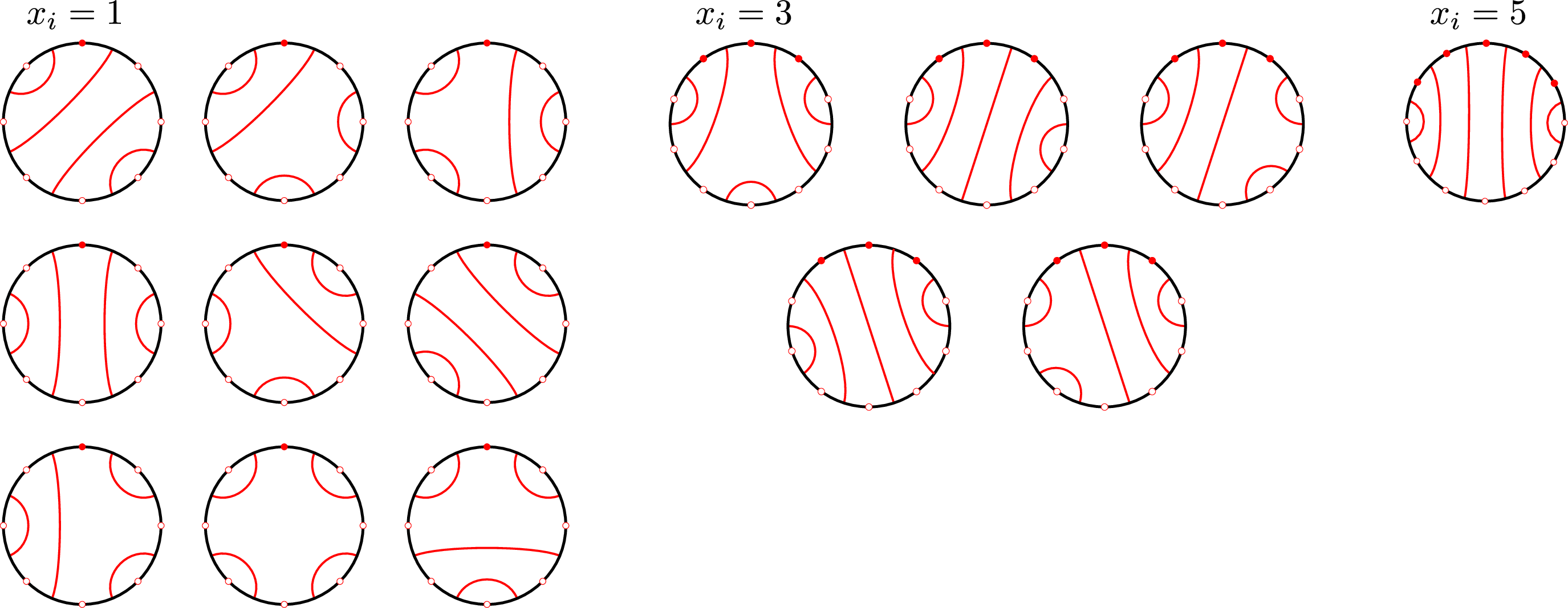}
 	\caption{Dividing set configurations on $D_i$. $x_i=2j-1$ denotes the number of middle intersection points on $\partial N_i$.}
 	\label{fig:osszes_lehetseges_osszekotes}
 \end{figure}
 
 Figure \ref{fig:osszes_lehetseges_osszekotes} shows fifteen structures of the dividing set on $D_i$ which we need to examine. However, after rounding the edges between $\partial N_i$ and $\partial \Bout$ we can show that there are nine of these configurations which give overtwisted contact structures on the 3-ball $\Bout \setminus \bigcup\limits_{i=1}^2 (N_i \cup D_i)$, thus can be excluded:
 
 \begin{lemma}
 	If the dividing set on $D_1$ or $D_2$ looks like one of the configurations of Figure \ref{fig:rossz_osszekotesek}, then the contact structure on $B^3$ is overtwisted.
 \end{lemma}
 \begin{figure}[h!]
	\centering
	\includegraphics[width=0.44\textwidth]{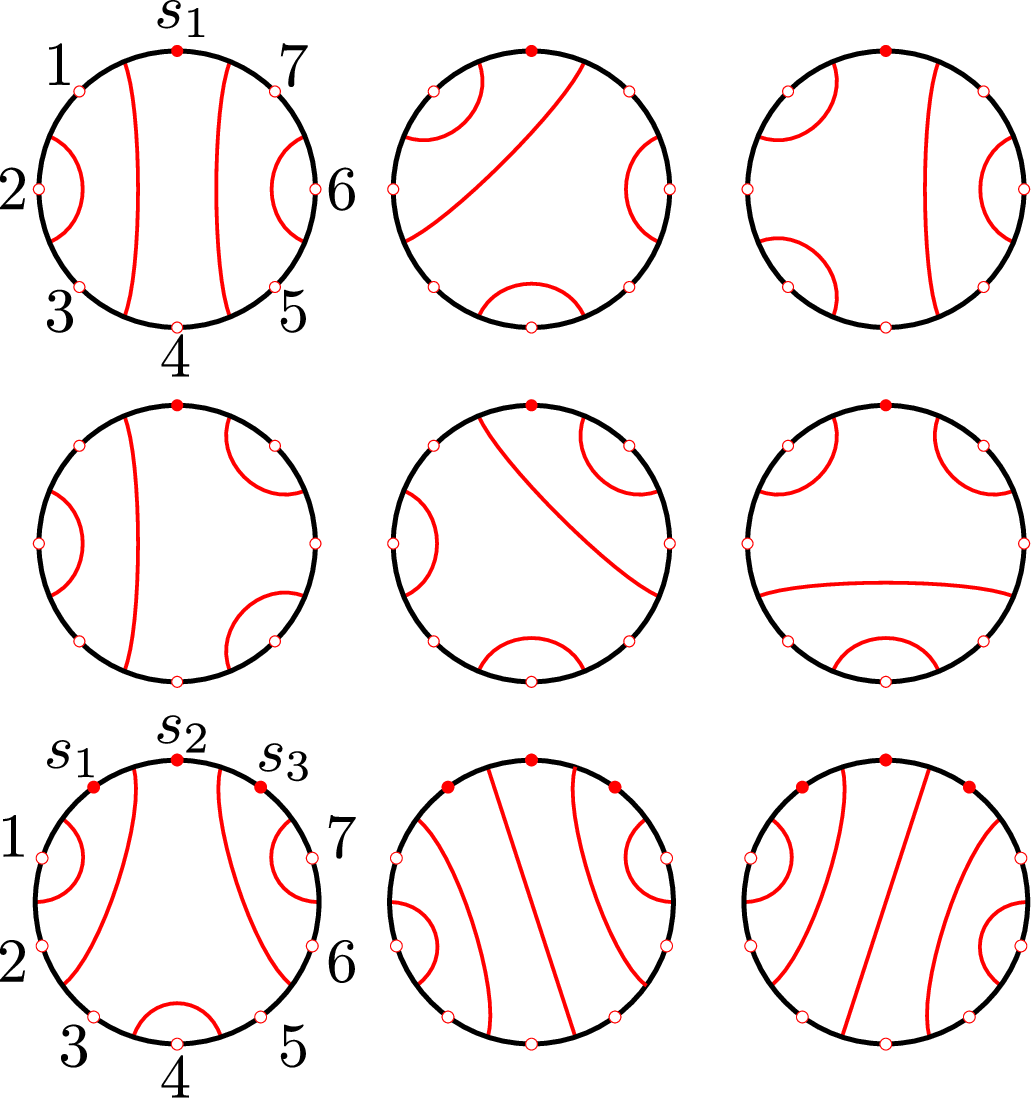}
	\caption{The dividing set configurations on $D_i$ which we excluded because they gave overtwisted contact structures. We numbered the intersection points of $\partial D_i$ and the dividing set. For example, in case $x_i=1$, the second diagram in the first row has three boundary parallel arcs bypassing intersection points 1, 4 and 6.}
	\label{fig:rossz_osszekotesek}
\end{figure}
\begin{proof}
		Number the intersection points of $\partial D_i$ and the dividing set as in \mbox{Figure \ref{fig:rossz_osszekotesek}}. To check that a construction gives overtwisted contact structure, we study the boundary parallel dividing arcs of it. 
		
		In Figure \ref{fig:Bp246_OT}, we examine the bypass of intersection point 2, 4 and 6, respectively. We will show that these are disallowed, because any of these three bypasses result in an overtwisted contact structure.
		
		Indeed, in Figure \ref{fig:Bp246_OT} we can see that after attaching the bypass along the blue arc and rounding the edges between $\partial N_i$ and $\partial \Bout$, the dividing set has a homotopically trivial component. Thus the contact structure near the surface is overtwisted.
		
		Now observe that each configuration of Figure \ref{fig:rossz_osszekotesek} has a boundary parallel arc bypassing intersection point number 2, 4 or 6. This means that they are all disallowed.
	\end{proof}
		
	\begin{figure}[!h]
		\centering
		\includegraphics[width=0.9\textwidth]{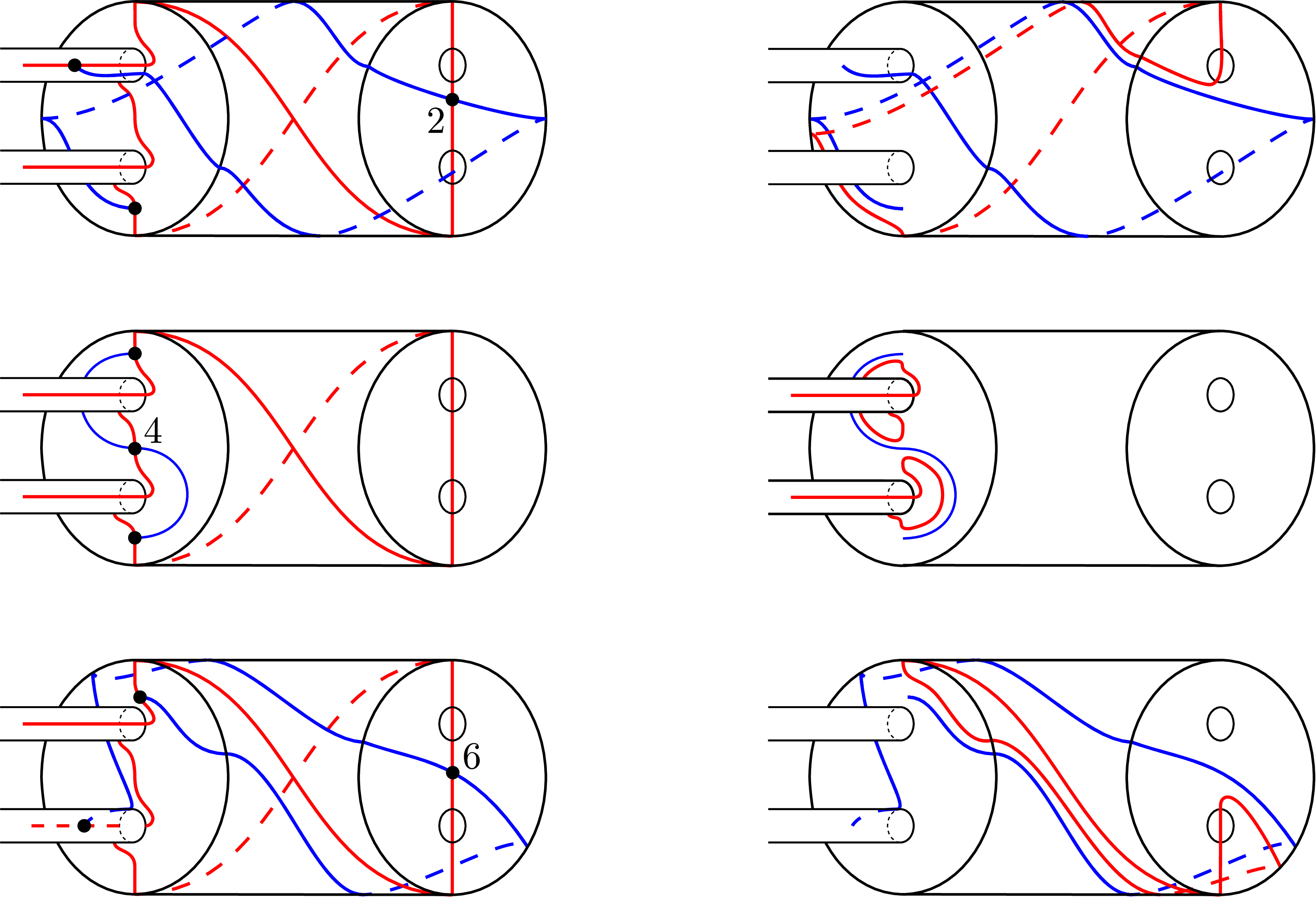}
		\caption{Attaching a bypass whose middle intersection point is 2, 4 or 6 gives an overtwisted contact structure. On the left we can see the original dividing curve and the attaching arc. On the right the homotopically trivial component of the obtained dividing set is shown. }
		\label{fig:Bp246_OT}
	\end{figure}

  The six remaining arrangements are the ones in Figure \ref{fig:osszekotesek}.

 \begin{figure}[h!]
 	\centering
 	\includegraphics[width=0.44\textwidth]{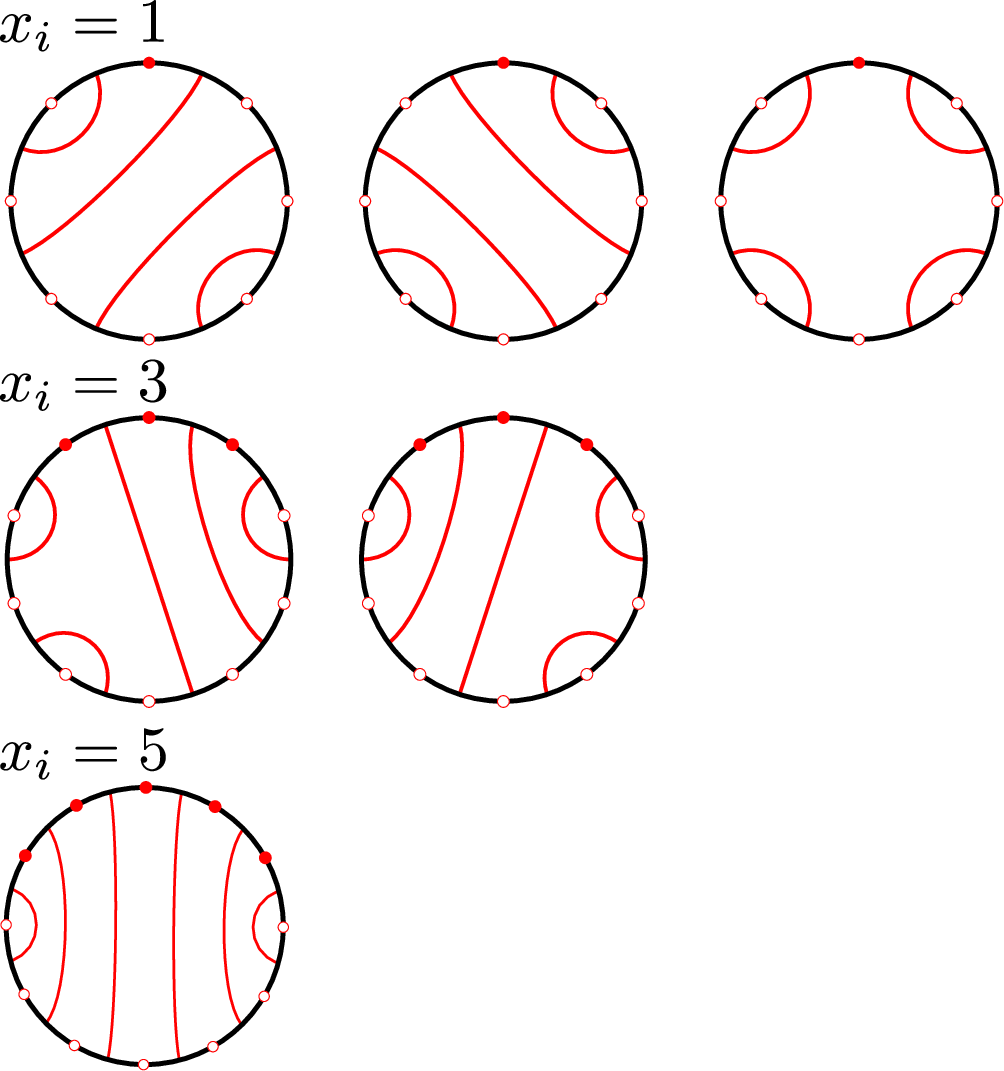}
 	\caption{Possible dividing set configurations on the disk $D_i$.}
 	\label{fig:osszekotesek}
 \end{figure}

In what follows, we show that we cannot choose every possible pair of arrangements in Figure \ref{fig:osszekotesek} to be the dividing set of $D_1$ and $D_2$. We have the following requirement.

\begin{lemma}\label{lem:szarakon_osszeg_allando}
	For $i=1,2$, let $\bar{x}_i=2j_i+1$ be the number of intersection
	points of $\partial D_i$ and the dividing curve on $\partial N_i$.
	Then $\bar{x}_1+\bar{x}_2=10$.
\end{lemma}
\begin{proof}
	Let $f_S$ denote the Seifert framing and $f_{TB}$ the
	Thurston--Bennequin framing of $\L$. We think of framings as
	longitudes on $\partial N$, oriented along $\L$, and write
	$\tw(f_1,f_2)$ for their signed intersection number, with the
	usual framing sign convention.
	
	Choose a smooth framing $f$ in the standard model so that
	$\restr{f}{\Bout}$ agrees with the framing induced by
	$\partial D_i$ along $l_i$, and extend it over $\Bin$.
	Transport this framing and the chosen disks to every realization,
	keeping the identifications at the four endpoints, the boundary
	contact germs and the rounding of the corners fixed. Then
	$\tw(f,f_S)$ is independent of the Legendrian realization, and
	$$\tb(\L)=\tw(f_{TB},f)+\tw(f,f_S).	$$
	
	Orient both dividing curves on $\partial N$ along $\L$.
	For a subarc $\L'\subset\L$, we define the \emph{relative twisting}
	$\rtw(\restr{f_{TB}}{\L'},\restr{f}{\L'})$ as one half of the
	signed intersection number of the two dividing arcs over $\L'$
	with the comparison longitude $\restr{f}{\L'}$, using the fixed
	endpoint identifications. For a union of subarcs we take the sum.
	Both dividing curves represent the contact framing, so
	$$	\tw(f_{TB},f)
	=\rtw(\restr{f_{TB}}{\Bout},\restr{f}{\Bout})
	+\rtw(\restr{f_{TB}}{\Bin},\restr{f}{\Bin}).$$
	
	Since $\L$ has maximal $\tb$, its inner part cannot destabilize.
	By Proposition \ref{prop:2-braid_XSZ}, it consists either of a
	fixed number of $X$ blocks or of a fixed number of $S$ and $Z$
	blocks. With the orientation inherited from $\L$, the $S$ and $Z$
	blocks have the same writhe and number of cusps. Hence the inner
	relative twisting is constant. Since $\tb(\L)$ and
	$\tw(f,f_S)$ are also fixed, the outer relative twisting is constant.
	
	Let $\rtw_i$ denote the outer relative twisting along $l_i$.
	By the efficient choice of the longitudinal parts and our relative
	convention,
	$$
	2|\rtw_i|=\bar{x}_i,\qquad
	\rtw(\restr{f_{TB}}{\Bout},\restr{f}{\Bout})
	=\rtw_1+\rtw_2.
	$$
	To determine the constant, use the maximal $\tb$ realization from
	Lemma \ref{lem:outside_max_tb}, whose outer part is shown in
	Figure \ref{fig:one_ruling}. Figure
	\ref{fig:szarakon_osszeg_allando} gives
	$\bar{x}_1=\bar{x}_2=5$. In the standard tube coordinates, the
	comparison longitudes cross the dividing arcs in the same winding
	direction on the two analogous local configurations. Hence, for this particular realization, the outer contributions
	have the same sign, so $\rtw_1=\rtw_2=\pm5/2$, and 
	$$	|\rtw_1+\rtw_2|=5	$$
	for every maximal $\tb$ realization.
	
	\begin{figure}[h!]
		\centering
		\includegraphics[width=0.85\textwidth]{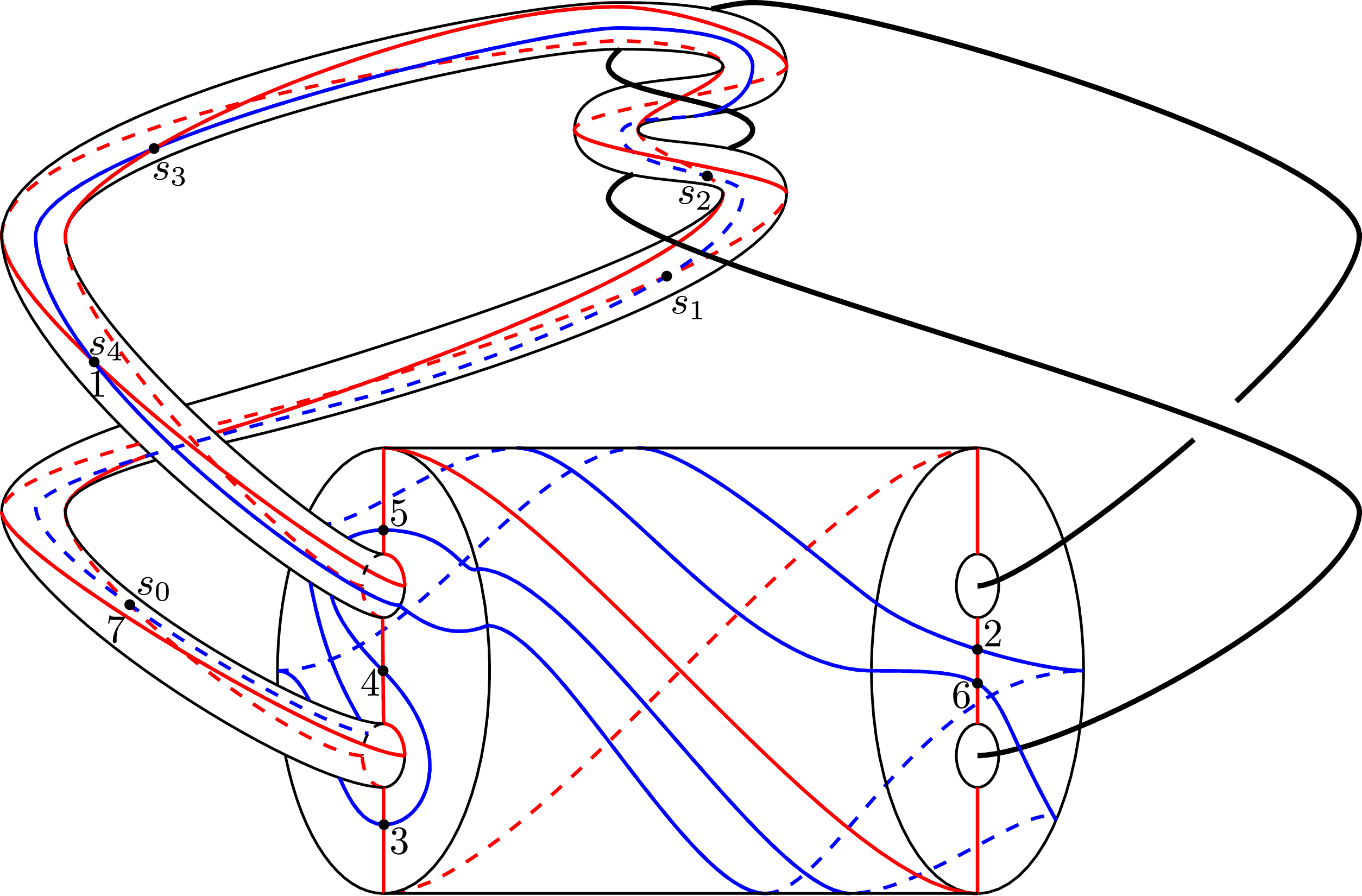}
		\caption{An illustration of the case $x_1=x_2=3$.
			The blue curve is $\partial D_1$. The second compressing disk
			$D_2$ is chosen by the analogous construction along $l_2$.}
		\label{fig:szarakon_osszeg_allando}
	\end{figure}
	
	It remains to show that the two outer contributions have the
	same sign for every maximal $\tb$ realization. Recall that $\bar{x}_i\leq7$, so $|\rtw_i|\leq7/2$.
	If the two relative twistings had opposite signs, then
	$|\rtw_1+\rtw_2|\leq7/2$, a contradiction. Therefore,
	$$
	\bar{x}_1+\bar{x}_2
	=2\bigl(|\rtw_1|+|\rtw_2|\bigr)
	=2|\rtw_1+\rtw_2|=10,
	$$
	which finishes the proof.
\end{proof}

We had six possible dividing set configurations for $D_i$: one with $x_i=5$, two with $x_i=3$ and three with $x_i=1$ (see Figure \ref{fig:osszekotesek}). According to Lemma \ref{lem:szarakon_osszeg_allando}, we have ten possible ways to pair these arrangements up to get dividing sets for $D_1$ and $D_2$, which finishes the proof of Theorem \ref{thm:max10_kivul}.
\end{proof}

We are now ready to prove Theorem \ref{thm:main}.

\begin{proof}[Proof of Theorem \ref{thm:main}]
	For $m=2$, the knot $K(4,2)$ is the twist knot $K_{-5}$ in the
	notation of Etnyre, Ng and V\'ertesi. By \cite{ENV}, it has exactly two non-destabilizable Legendrian representatives, so the claimed	upper bound holds. Thus, we may assume $m>2$ or $m\leq -2$.
	
	By the proof of Theorem \ref{thm:max10_kivul}, there are at most $10$ possible ordered pairs of dividing set configurations for $D_1$ and $D_2$.
	However, when counting closed Legendrian realizations, some of these cases can be identified using the symmetry of the standard model.
	
	Consider the rotation
	$\varrho(x,y,z)=(-x,-y,z)$
	by $180^\circ$ about the $z$ axis. Since $$	\varrho^*(dz-y\,dx)=dz-y\,dx,
	$$
	$\varrho$ is a contactomorphism of $(S^3,\xi_{std})$. The rotation
	interchanges the two outer strands and hence the roles of $D_1$ and
	$D_2$.
	
	The three pairs with $(x_1,x_2)=(1,5)$ are therefore identified with
	the three pairs with $(x_1,x_2)=(5,1)$, giving three cases up to this
	symmetry. If $(x_1,x_2)=(3,3)$, there are two possible dividing set
	configurations on each disk. The rotation induces an involution on
	this two-element set, and in either case the four ordered pairs form
	three orbits. Thus the ten possible ordered pairs form at most six
	orbits under $\varrho$.
	
	For a closed Legendrian realization $\L$, the knots $\L$ and
	$\varrho(\L)$ are Legendrian isotopic. Indeed, $\varrho$ is a contactomorphism, and Proposition \ref{prop:elegctctotszamolni} implies that their contactomorphic complements determine the same Legendrian isotopy class. Hence, it is enough to consider one representative from each of the six orbits above.

	For each of these outer contact structures, the number of
	possible inner Legendrian pieces is bounded by
	Theorem \ref{thm:belul}.
	The Legendrian isotopies used there are supported away
	from $S$, so they extend to contact isotopies fixed near $P$.
	If two maximal $\tb$ realizations determine the same outer
	contact structure and the same inner Legendrian isotopy
	class, the corresponding contactomorphisms agree near $P$
	and therefore glue to a contactomorphism of the complements
	of the resulting closed Legendrian knots.
	By Proposition \ref{prop:elegctctotszamolni}, the two knots
	are Legendrian isotopic.
	
	 Therefore, the number
	of maximal $\tb$ Legendrian realizations is at most $6$ if $m>2$,
	and at most
	$$6\cdot\left\lfloor\frac{(|m|+2)^2}{4}\right\rfloor \qquad \textrm{if } m\leq-2.$$
\end{proof}

\begin{proof}[Proof of Corollary \ref{cor:oriented}]
	For $m=2$, the knot $K(4,2)$ is $K_{-5}$ in the notation
	of \cite{ENV}, which has two oriented Legendrian isotopy
	classes of maximal $\tb$ representatives. Thus the bound
	holds, and we may assume $m>2$ or $m\leq-2$.
	
	Write
	$$
	N_m=
	\begin{cases}
		1, & \textrm{if } m>2,\\[1mm]
		\left\lfloor\dfrac{(|m|+2)^2}{4}\right\rfloor,
		& \textrm{if } m\leq-2.
	\end{cases}
	$$
	By Theorem \ref{thm:belul}, there are at most $N_m$
	possible inner pieces. The isotopies used in its proof
	are fixed near $S$, so this bound also applies when
	the orientations at the endpoints are fixed.
	
	Keep the four endpoints marked as in the standard smooth
	model. Record $\varepsilon=+1$ if the oriented knot enters
	$\Bin$ at the upper right endpoint, and $\varepsilon=-1$
	if it leaves there. Since the closed knot has one component,
	this determines the orientations at all four endpoints.
	Together with the 10 ordered outer configurations counted
	in the proof of Theorem \ref{thm:max10_kivul}, we therefore
	have at most $20N_m$ complete choices.
	
	Each complete choice determines at most one oriented
	Legendrian isotopy class. Indeed, the gluing argument in
	the proof of Theorem \ref{thm:main} gives a contactomorphism
	of the complements when the outer configuration and the
	inner piece agree. The comparison maps respect the fixed
	boundary markings, and agreement of $\varepsilon$ ensures
	that the oriented Seifert longitude is preserved.
	The contactomorphism therefore extends over the standard
	neighborhoods preserving the orientations of their cores.
	As in the proof of Proposition \ref{prop:elegctctotszamolni},
	this gives an oriented Legendrian isotopy.
	
	Consider the rotation
	$\varrho(x,y,z)=(-x,-y,z)$ used in the proof of
	Theorem \ref{thm:main}. It is contact isotopic to the
	identity in the standard $\mathbb{R}^3$ chart. Indeed,
	putting $u=z-xy/2$, the contact form becomes
	$$	dz-y\,dx=du+\frac12(x\,dy-y\,dx),$$
	and rotations of the $(x,y)$-plane with $u$ fixed join
	the identity to $\varrho$. Thus $\L$ and $\varrho(\L)$
	are oriented Legendrian isotopic when the orientation
	is transported by $\varrho$.
	
	The rotation interchanges the two outer strands.
	To determine its effect on $\varepsilon$, take
	$\varepsilon=+1$. At the right endpoints the knot enters
	$\Bin$ above and leaves below. At the left endpoints,
	it leaves above and enters below when $m$ is even,
	whereas it enters above and leaves below when $m$ is odd.
	Hence $\varrho$ changes $\varepsilon$ to $-\varepsilon$
	in the even case and preserves $\varepsilon$ in the
	odd case.
	
	For negative $m$, read the inner word from left to right.
	The rotation reverses its order and interchanges $S$
	and $Z$. Applying $\varrho$ to the local isotopy $SSZ\simeq ZSS$
	gives the local isotopy $SZZ\simeq ZZS$. Hence the rotation
	induces an involution on the inner parameters used in the
	proof of Theorem \ref{thm:belul}.
	
	If $m$ is even, the involution on the complete choices
	has no fixed points, since it changes $\varepsilon$.
	Thus the at most $20N_m$ choices form at most $10N_m$
	orbits, giving the claimed bounds for even $m$.
	
	If $m>2$ is odd, then $N_m=1$ and the rotation preserves
	$\varepsilon$. For each of its two values, the 10 outer
	configurations form at most 6 orbits, as in the proof
	of Theorem \ref{thm:main}. Thus there are at most 12
	oriented Legendrian isotopy classes.
	
	Finally, suppose that $m\leq-3$ is odd, and write
	$|m|=2q+1$. Let $a$ and $b$ be the numbers of $S$ blocks
	at odd and even indices, respectively, as in the proof
	of Theorem \ref{thm:belul}. Thus
	$0\leq a\leq q+1$ and $0\leq b\leq q$.
	Reversal preserves the parity of the indices in a word
	of odd length, so the rotation acts by
	$$	(a,b)\longmapsto(q+1-a,q-b).$$
	A fixed pair would satisfy $2(a+b)=2q+1$, which is
	impossible. Hence the involution on the complete choices
	again has no fixed points. Their number of orbits is
	therefore at most $10N_m$, proving the remaining case.
\end{proof}

\section*{Acknowledgement}
The first author was supported by the EKÖP-26 University Research Scholarship Program of the Ministry of Science and Technology from the source of the National Research, Development and Innovation Fund (grant no. EKÖP-26-4-II-ELTE-20260370). All mathematical results, calculations, figures and examples were developed by the authors. OpenAI's ChatGPT was used during manuscript revision for language editing and typo detection.

\bibliographystyle{plain}
\bibliography{biblio}
\end{document}